\documentclass[10pt,a4paper]{article}
\usepackage{commath}
\usepackage{graphicx}
\usepackage{enumitem}
\usepackage{float}
\usepackage{amsthm,amssymb,mathrsfs,amsmath}

\newcommand\bldxi{\boldsymbol\omega}
\newcommand\bldx{\boldsymbol t}

\newcommand\bldt{\boldsymbol x}

\newcommand\xx{t}
\newcommand\yy{y}

\newcommand\tee{x}
\newcommand\xxi{\omega}

\newcommand\PQRS{\alpha_2}
\newcommand\SRQP{\beta_2}
\newcommand\UVWX{\alpha_1}
\newcommand\XWVU{\beta_1}
\newcommand\MNOP{\gamma_2}
\newcommand\PONM{\gamma_1}
\newcommand\ABCD{\delta_2}
\newcommand\DCBA{\delta_1}

\usepackage[a4paper,bindingoffset=0.2in,left=0.5in,right=0.5in,top=1in,bottom=1in,footskip=.25in]{geometry}
\usepackage{cite}
\usepackage[colorlinks,citecolor=blue,urlcolor=blue,bookmarks=false,hypertexnames=true]{hyperref}
\hypersetup{citecolor=blue}
\newtheorem{definition}{Definition}[section]
\newtheorem{theorem}{Theorem}[section]

\newtheorem{lemma}[theorem]{Lemma}

\newtheorem{remark}{Remark}

\title{Uncertainty principles associated with the short time quaternion coupled fractional Fourier transform}
\author{Bivek Gupta$^a$, Amit K. Verma$^b$, Ravi P. Agarwal$^c$\thanks{$^a$bivekgupta040792@gmail.com,$^b$akverma@iitp.ac.in, Corresponding Author: $^c$Ravi.Agarwal@tamuk.edu}\\\small{\it{$^{a,b}$ Department of Mathematics,}} \\\small{\it{Indian Institute of Technology Patna,}}\\\small{\it{ Bihta, Patna 801103, (BR) India.}}\\\small{\it{$^c$Department of Mathematics, Texas A\&M University-Kingsville, 700 University Blvd., }}\\\small{\it{MSC 172, Kingsville, Texas  78363-8202.}}}
\date{\today}

\begin{document}
\maketitle
\begin{abstract}
In this paper, we extend the coupled fractional Fourier transform of a complex valued functions to that of the quaternion valued functions on $\mathbb{R}^4$ and call it the quaternion coupled fractional Fourier transform (QCFrFT). We obtain the sharp Hausdorff-Young inequality for QCFrFT and obtain the associated R\`enyi uncertainty principle. We also define the short time quaternion coupled fractional Fourier transform (STQCFrFT) and explore  its important properties followed by the Lieb's and entropy uncertainty principles.
\end{abstract}

\section{Introduction}

Based on the knowledge that the Hemite functions are the eigen functions of the Fourier transform (FT) with eigen values $e^{in\frac{\pi}{2}},$ Namias \cite{namias1980fractional} in 1980 introduced the fractional Fourier transform (FrFT) with angle $\theta$ as an integral transform  whose eigen function are the Hermite functions but with eigen values $e^{in\theta},$ and which reduces  to the FT when $\theta=\frac{\pi}{2}.$ This later was refined by McBride and Kerr \cite{kerr1988namias}, \cite{kerr1988distributional}. The extension of the FrFT to the higher dimension can be seen in \cite{verma2021note}, where the kernel of the transform has been obtained by taking the tensor product of $n$ copies of the kernel of one-dimensional transform. Following the ideas of the Namias's  and using the fact the the Hermite function of two complex variables are eigenfunction of the two dimensional FT, Zayed in  \cite{zayed2018two},\cite{zayed2019new} introduced a new definition of  two dimensional FrFT $\mathcal{F}^{\alpha,\beta},$ which is not a tensor product of two copies of one-dimensional transform and is given as 
\begin{align}\label{P7eqn1}
(\mathcal{F}^{\alpha,\beta}f)(\boldsymbol u)=\tilde{d}(\gamma)\int_{\mathbb{R}^2}f(\bldt)e^{-i\{\tilde{a}(\gamma)(|\bldt|^2+|\boldsymbol u|^2)-\bldt\cdot M\boldsymbol u\}}d\bldt,
\end{align}
where $\alpha,\beta\in \mathbb{R}$ and are such that $\alpha+\beta\notin 2\pi\mathbb{Z}$ and $\gamma=\frac{\alpha+\beta}{2},~\delta=\frac{\alpha-\beta}{2}, ~\tilde{a}(\gamma)=\frac{\cos\gamma}{2},~\tilde{b}(\gamma,\delta)=\frac{\cos\delta}{\sin\gamma},~\tilde{c}(\gamma,\delta)=\frac{\sin\delta}{\sin\gamma},~\tilde{d}(\gamma)=\frac{ie^{-i\gamma}}{2\pi\sin\gamma},$  
$M=
\begin{pmatrix}
\tilde{b}(\gamma,\delta) & \tilde{c}(\gamma,\delta)\\
-\tilde{c}(\gamma,\delta) & \tilde{b}(\gamma,\delta)
\end{pmatrix}.
$
Furthermore, the transform depends on the angles $\alpha$ and $\beta$ that are coupled so that the transform parameters are $\gamma=\frac{\alpha+\beta}{2}$ and $\delta=\frac{\alpha-\beta}{2}.$ Kamalakkannan et al. \cite{kamalakkannan2022extension} proved the Parseval's identity, inversion theorem and that the class $\{\mathcal{F}^{\alpha,\beta}:\alpha,\beta\in\mathbb{R},~\alpha+\beta\notin2\pi\mathbb{Z}\}$ is a family of unitary operators  on $L^2(\mathbb{R}^2)$ which satisfies the additive property $\mathcal{F}^{\alpha',\beta'}(\mathcal{F}^{\alpha,\beta}f)=\mathcal{F}^{\alpha+\alpha',\beta+\beta'}f$ for $\alpha+\beta,~\alpha'+\beta',~\alpha+\beta+\alpha'+\beta'\notin 2\pi\mathbb{Z}.$ Kamalakkannan et al. \cite{kamalakkannan2020multidimensional} extended the fractional Fourier transform (FrFT)  to the $n-$dimensional FrFT, which is more general than that in \cite{verma2021note}, and introduced a corresponding convolution structure followed by the convolution theorem. Recently, Shah et al. \cite{shah2022class} obtained the Heisenberg uncertainty principle (UP) followed by the local and logarithmic UPs. They also established some concentration based UP including Amrein-Berthier-Bebedicks, Donoho-Stark's UPs etc.

Even though CFrFT generalizes the FrFT to two dimension, but because of the presence of global kernel it fails in giving the local information of non-transient signals. Thus, Kamalakkannan et al. (\cite{kamalakkannan2021short}) developed a short time coupled fractional Fourier transform (STCFrFT) and obtained the associated Parseval's and inversion formula followed by some associated UPs.
 
The quaternion Fourier transform (QFT), introduced by Ell \cite{ell1993quaternion}, is useful in the analysis of $\mathbb{H}-$valued, i.e., quaternion valued functions. Based on the non-commutativity of the product of quaternion, the QFT can be classified into various types, namely, left-sided, right-sided and two-sided \cite{bahri2014continuous},\cite{bahri2008uncertainty},\cite{ell1993quaternion}. For the right sided QFT, Cheng et. al \cite{cheng2019plancherel} discussed the Plancherel theorem and also obtained its relation  with the other two QFT. Lian \cite{lian2018uncertainty}, proved the Pitt's inequality, logarithmic UP (also see \cite{chen2015pitt}), entropy UP for the two-sided QFT with optimal constants. Also, in \cite{lian2020sharp}, author obtained the sharp Hausdorff-Young (H-Y) inequality along with the  Hirschman's entropy UP for the two-sided QFT using the standard differential approach. Recently, QFT has been extended to the quaternion FrFT (QFrFT) and also the quaternion quadratic phase Fourier transform (QQPFT)\cite{bahri2023some},\cite{gupta2022short}. Replacing the kernels in the definition of the two sided QFT (\cite{lian2021quaternion}) of the function defined on $\mathbb{R}^2$, with that of the kernels of the FrFT (\cite{namias1980fractional},\cite{almeida1994fractional},\cite{verma2021note}) results in the two-sided QFrFT. Similar to this, both other variant of QFrFT can be found in \cite{wei2013different}.

The generalization of classical windowed Fourier transform to $\mathbb{H}-$valued functions on $\mathbb{R}^2$ can be found in \cite{bahri2010windowed}. Authors obtained its several important properties  using the properties of the right sided QFT \cite{bahri2008uncertainty}. They also obtained the Heisenberg UP for the quaternion windowed Fourier transform (QWFT) using the same technique as Wilczok \cite{wilczok2000new} used. In \cite{bahri2020uncertainty}, authors also obtained the Pitt's inequality and Lieb's inequality  for the right sided QWFT introduced in \cite{bahri2010windowed}. Including, the orthogonality property for the two sided QWFT, authors in \cite{kamel2019uncertainty},\cite{brahim2020uncertainty} studied several UPs including Beckner's UP in terms of entropy and Lieb's UP. Substituting the Fourier kernel in the left sided, right sided or two sided QWFT by the fractional Fourier kernels, results in the corresponding QWFrFT. 

As noted above, some important properties as well as the UPs of the CFrFT and the STCFrFT have been investigated for the complex valued function. It is natural to extend these transforms to quaternion setting. As far as we know these transforms have not been introduced  for the $\mathbb{H}-$valued functions. This paper deals with the two-sided QCFrFT. We derive the sharp H-Y inequality for QCFrFT followed by the R\`enyi entropy UP for QCFrFT. We also introduce the STQCFrFT and along with its basic properties like, linearity, translation etc., we also obtain its inner product relation and reconstruction formula. The Lieb's and entropy UPs for the proposed STQCFrFT is obtained with the aid of QCFrFT's sharp H-Y inequality.

The paper is organized as follows:
In section 2, we recall some basics associated with the quaternion algebra. We define the two sided QCFrFT and establish its several important properties in section 3. In section 4, we defined the two sided STQCFrFT and explore its properties along with the Lieb's and entropy UPs followed by an example of the STQCFrFT. Lastly, in section 5, we conclude this paper.

\section{Preliminaries}
The quaternion algebra $\mathbb{H}=\{s=s_0+is_1+js_2+ks_3 :s_0,s_1,s_2,s_3\in\mathbb{R}\},$ where $i,j$ and $k$ are the imaginary units satisfying the Hamilton's multiplication rule $i^2=j^2=k^2=-1,~ij=-ji=k,~jk=-kj=i,~ki=-ik=j$. For a quaternion $s=s_0+is_1+js_2+ks_3,$ we define the following terms
\begin{itemize}
\item $Sc(s)=s_0,$ called the scalar part of $s$ which satisfies the cyclic multiplicative symmetry (\cite{hitzer2007quaternion}), i.e., $Sc(qrs)=Sc(sqr)=Sc(rsq),~\forall~q,r,s\in\mathbb{H}.$
\item the quaternion conjugate $\bar{s}=s_0-is_1-js_2-ks_3,$ which satisfies $\overline{rs}=\bar{r}\bar{s},~\overline{r+s}=\bar{s}+\bar{s},~\bar{\bar{s}}=s,~\forall~r,s\in\mathbb{H}.$
\item the modulus $|s|=\sqrt{s\bar{s}}=\left(\sum_{l=0}^3s_l^2\right)^{\frac{1}{2}},$ which satisfies $|rs|=|r||s|,~\forall~r,s\in\mathbb{H}.$
\end{itemize}
A $\mathbb{H}-$valued function $g$ defined on $\mathbb{R}^n$  is of the form $g(\bldx)=g_0(\bldx)+ig_1(\bldx)+jg_2(\bldx)+kg_3(\bldx),~\bldx\in\mathbb{R}^n,$
where $g_0,g_1,g_2$ and $g_3$ are real valued. The $L^q-$norm, $1\leq q<\infty,$ of $g$ is defined by
\begin{align}\label{P7eqn6}
\|g\|_{L^q_\mathbb{H}(\mathbb{R}^n)}
&=\left(\int_{\mathbb{R}^n}|g(\bldx)|^qd\bldx\right)^\frac{1}{q}
\end{align}
and the collection of all measurable function with $\mathbb{H}-$valued having finite $L^q-$norm is a Banach space denoted by $L^q_\mathbb{H}(\mathbb{R}^n).$ $L^\infty_\mathbb{H}(\mathbb{R}^n)$ is the set of all essentially bounded measurable $\mathbb{H}-$valued functions with norm 
\begin{align}\label{P7eqn7}
\|g\|_{L^\infty_\mathbb{H}(\mathbb{R}^n)}=\mbox{ess~sup}_{\bldx\in\mathbb{R}^n}|g(\bldx)|.
\end{align}
Moreover, the $\mathbb{H}-$valued inner product 
\begin{align}\label{P7eqn8}
(g,h)=\int_{\mathbb{R}^n}g(\bldx)\overline{h(\bldx)}d\bldx,
\end{align}
with symmetric real scalar part $\langle g,h\rangle=Sc\left(\int_{\mathbb{R}^n}g(\bldx)\overline{h(\bldx)}d\bldx\right)$
turns $L^2_\mathbb{H}(\mathbb{R}^n)$ to a Hilbert space, where the norm in \eqref{P7eqn6} can be expressed as 
\begin{align}
\|g\|_{L^2_\mathbb{H}(\mathbb{R}^n)}=\sqrt{\langle g,g \rangle}=\sqrt{(g,g)}=\left(\int_{\mathbb{R}^n}|g(\bldx)|^2d\bldx\right)^\frac{1}{2}.
\end{align}

\section{Quaternion Coupled Fractional Fourier transform (QCFrFT)}
In this section we give a definition of two sided QCFrFT and study its important properties.
\begin{definition}\label{P7Defn3.1}
Let $\boldsymbol\alpha=(\UVWX,\PQRS),\boldsymbol\beta=(\XWVU,\SRQP) \in\mathbb{R}^2$ such that $\UVWX+\XWVU,~\PQRS+\SRQP\notin 2\pi\mathbb{Z} $. The QCFrFT of $f(\bldt)\in L^2_{\mathbb{H}}(\mathbb{R}^4),~\bldt=(\bldt_1,\bldt_2)\in\mathbb{R}^2\times\mathbb{R}^2,$ is defined by 
\begin{align}\label{P7eqn11}
(\mathcal{F}^{\boldsymbol\alpha,\boldsymbol\beta}_{\mathbb{H}} f)(\bldxi)=\int_{\mathbb{R}^4}\mathcal{K}^i_{\UVWX,\XWVU}(\bldt_1,\bldxi_1)f(\bldt)\mathcal{K}^j_{\PQRS,\SRQP}(\bldt_2,\bldxi_2)d\bldt,~\bldxi=(\bldxi_1,\bldxi_2)\in\mathbb{R}^2\times\mathbb{R}^2
\end{align}
where 
\begin{align}\label{P7eqn12}
\mathcal{K}^i_{\UVWX,\XWVU}(\bldt_1,\bldxi_1)=\tilde{d}(\PONM)e^{-i\left\{\tilde{a}(\PONM)(|\bldt_1|^2+|\bldxi_1|^2)-\bldt_1\cdot M_1\bldxi_1\right\}}
\end{align}
and 
\begin{align}\label{P7eqn13}
\mathcal{K}^j_{\PQRS,\SRQP}(\bldt_2,\bldxi_2)=\tilde{d}(\MNOP)e^{-j\left\{\tilde{a}(\MNOP)(|\bldt_2|^2+|\bldxi_2|^2)-\bldt_2\cdot M_2\bldxi_2\right\}},
\end{align}
with $\PONM=\frac{\UVWX+\XWVU}{2},~\DCBA=\frac{\UVWX-\XWVU}{2}, ~\tilde{a}(\PONM)=\frac{\cos\PONM}{2},~\tilde{b}(\PONM,\DCBA)=\frac{\cos\DCBA}{\sin\PONM},~\tilde{c}(\PONM,\DCBA)=\frac{\sin\DCBA}{\sin\PONM},~\tilde{d}(\PONM)=\frac{ie^{-i\PONM}}{2\pi\sin\PONM},$  
$M_1=
\begin{pmatrix}
\tilde{b}(\PONM,\DCBA) & \tilde{c}(\PONM,\DCBA)\\
-\tilde{c}(\PONM,\DCBA) & \tilde{b}(\PONM,\DCBA)
\end{pmatrix}
$
and 
$\MNOP=\frac{\PQRS+\SRQP}{2},~\ABCD=\frac{\PQRS-\SRQP}{2}, ~\tilde{a}(\MNOP)=\frac{\cos\MNOP}{2},~\tilde{b}(\MNOP,\ABCD)=\frac{\cos\ABCD}{\sin\MNOP},~\tilde{c}(\MNOP,\ABCD)=\frac{\sin\ABCD}{\sin\MNOP},~\tilde{d}(\MNOP)=\frac{je^{-j\MNOP}}{2\pi\sin\MNOP},$  
$M_2=
\begin{pmatrix}
\tilde{b}(\MNOP,\ABCD) & \tilde{c}(\MNOP,\ABCD)\\
-\tilde{c}(\MNOP,\ABCD) & \tilde{b}(\MNOP,\ABCD)
\end{pmatrix}.
$
\end{definition}
The corresponding inversion formula is given by
\begin{align}\label{P7eqn14}
f(\bldt)=\int_{\mathbb{R}^4}\overline{\mathcal{K}^i_{\UVWX,\XWVU}(\bldt_1,\bldxi_1)}(\mathcal{F}^{\boldsymbol\alpha,\boldsymbol\beta}_{\mathbb{H}} f)(\bldxi)\overline{\mathcal{K}^j_{\PQRS,\SRQP}(\bldt_2,\bldxi_2)}d\bldxi
\end{align}\label{P7Remark_Particular-Case-of-QCFrFT}
\begin{remark} For $\boldsymbol\alpha=\boldsymbol\beta,$ the kernels $\mathcal{K}^i_{\UVWX,\UVWX}(\bldt_1,\bldxi_1)$ and $\mathcal{K}^j_{\PQRS,\PQRS}(\bldt_2,\bldxi_2)$ are the tensor product of two one-dimensional FrFT kernel. Thus, the  QCFrFT reduces to the two sided QFrFT of the function $f\in L^2_{\mathbb{H}}(\mathbb{R}^4).$ Moreover, if $\boldsymbol\alpha=\boldsymbol\beta=(\frac{\pi}{2},\frac{\pi}{2}),$ the QCFrFT reduces to the two sided QFT of the function $f\in L^2_{\mathbb{H}}(\mathbb{R}^4).$
\end{remark}
\subsection{QCFrFT in terms of QFT}
We now obtain an important relation between QCFrFT and QFT.
\begin{align*}
(\mathcal{F}^{\boldsymbol\alpha,\boldsymbol\beta}_{\mathbb{H}} f)(\bldxi)&=\int_{\mathbb{R}^4}\mathcal{K}^i_{\UVWX,\XWVU}(\bldt_1,\bldxi_1)f(\bldt)\mathcal{K}^j_{\PQRS,\SRQP}(\bldt_2,\bldxi_2)d\bldt\\
&=\tilde{d}_0(\PONM)e^{-i\tilde{a}(\PONM)|\bldxi_1|^2}\left\{\int_{\mathbb{R}^4} \frac{1}{2\pi}e^{i\bldt_1\cdot M_1\bldxi_1}\tilde{f}(\bldt)\frac{1}{2\pi}e^{j\bldt_2\cdot M_2\bldxi_2}d\bldt\right\}\tilde{d}_0(\MNOP)e^{-j\tilde{a}(\MNOP)|\bldxi_2|^2},
\end{align*}
where $\tilde{d}_0(\PONM)=\frac{ie^{-i\PONM}}{\sin\PONM},$ $\tilde{d}_0(\MNOP)=\frac{je^{-j\MNOP}}{\sin\MNOP}$ and 
\begin{align}\label{P7eqn15}
\tilde{f}(\bldt)=e^{-i\tilde{a}(\PONM)|\bldt_1|^2}f(\bldt)e^{-j\tilde{a}(\MNOP)|\bldt_2|^2}.
\end{align}
Thus,
\begin{align}\label{P7eqn16}
(\mathcal{F}^{\boldsymbol\alpha,\boldsymbol\beta}_{\mathbb{H}} f)(\bldxi)&=\tilde{d}_0(\PONM)e^{-i\tilde{a}(\PONM)|\bldxi_1|^2}\left(\mathcal{F}_{\mathbb{H}}\tilde{f}\right)(-M_1\bldxi_1,-M_2\bldxi_2)\tilde{d}_0(\MNOP)e^{-j\tilde{a}(\MNOP)|\bldxi_2|^2}
\end{align}
where
\begin{align}\label{P7eqn17}
\left(\mathcal{F}_{\mathbb{H}}\tilde{f}\right)(\bldxi)=\int_{\mathbb{R}^4}\frac{1}{2\pi}e^{-i\bldt_1\cdot\bldxi_1}\tilde{f}(\bldt)\frac{1}{2\pi}e^{-j\bldt_2\cdot\bldxi_2}d\bldt
\end{align}
is the quaternion Fourier transform of the function in $\mathbb{R}^4.$ \\

\textbf{Example of QCFrFT:} 

With the assumption that $\boldsymbol\alpha,~\boldsymbol\beta\in\mathbb{R}^2$ satisfies the conditions in definition \ref{P7Defn3.1}. We see that the QCFrFT $\mathcal{F}^{\boldsymbol\alpha,\boldsymbol\beta}_{\mathbb{H}} $ of the function
$f(\bldt)=e^{i\tilde{a}(\PONM)|\bldt_1|^2}e^{-A|\bldt_1|^2}e^{-B|\bldt_2|^2}e^{j\tilde{a}(\PONM)|\bldt_2|^2},~A,B>0$ is given as 
$$\left(\mathcal{F}^{\boldsymbol\alpha,\boldsymbol\beta}_{\mathbb{H}} f\right)(\bldxi)=\frac{1}{4AB}\tilde{d}_0(\PONM)e^{-i\tilde{a}(\PONM)|\bldxi_1|^2}e^{-\frac{1}{4}\left(\frac{|M_1\bldxi_1|^2}{A}+\frac{|M_2\bldxi_2|^2}{B}\right)}\tilde{d}_0(\MNOP)e^{-j\tilde{a}(\MNOP)|\bldxi_2|^2},$$
where $M_1$ and $M_2$ are the matrices given in definition \ref{P7Defn3.1}.\\

We now obtain the following important inequality, called the H-Y inequality,   based on the relation \eqref{P7eqn16} among QCFrFT and the QFT.
\begin{theorem}\label{P7Theo3.1} Let $f\in L^p_\mathbb{H}(\mathbb{R}^4)$ and $1\leq p\leq 2,$ $\frac{1}{p}+\frac{1}{q}=1,$ then
\begin{align}\label{P7eqn18}
\|\mathcal{F}^{\boldsymbol\alpha,\boldsymbol\beta}_{\mathbb{H}} f\|_{L^q_\mathbb{H}(\mathbb{R}^4)}\leq \left|\sin\left(\frac{\UVWX+\XWVU}{2}\right)\right|^{\frac{2}{q}-1}\left|\sin\left(\frac{\PQRS+\SRQP}{2}\right)\right|^{\frac{2}{q}-1}\|f\|_{L^p_\mathbb{H}(\mathbb{R}^4)},
\end{align}
where $A_p=\left(\frac{p^{\frac{1}{p}}}{q^{\frac{1}{q}}}\right)^\frac{1}{2}.$
\end{theorem}
\begin{proof}
Using relation \eqref{P7eqn16}, we get
\begin{align*}
\|\mathcal{F}^{\boldsymbol\alpha,\boldsymbol\beta}_{\mathbb{H}} f\|_{L^q_\mathbb{H}(\mathbb{R}^4)}
&=|\tilde{d}_0(\PONM)||\tilde{d}_0(\MNOP)|\left(\int_{\mathbb{R}^4}\left|\left(\mathcal{F}_\mathbb{H}\tilde{f}\right)(-M_1\bldxi_1,-M_2\bldxi_2)\right|^qd\bldxi\right)^\frac{1}{q}\\
&=\frac{|\tilde{d}_0(\PONM)||\tilde{d}_0(\MNOP)|}{\left(|\det(-M_1)||\det(-M_2)|\right)^{\frac{1}{q}}}\|\mathcal{F}_\mathbb{H}\tilde{f}\|_{L^q_\mathbb{H}(\mathbb{R}^4)}.
\end{align*}
Applying the sharp Hausdorff-Young inequality (\cite{lian2020sharp}) for the QFT, yields
\begin{align*}
\|\mathcal{F}^{\boldsymbol\alpha,\boldsymbol\beta}_{\mathbb{H}} f\|_{L^q_\mathbb{H}(\mathbb{R}^4)}&\leq \frac{|\tilde{d}_0(\PONM)||\tilde{d}_0(\MNOP)|A_p^4}{\left(|\det(-M_1)||\det(-M_2)|\right)^\frac{1}{q}}\|\tilde{f}\|_{L^p_\mathbb{H}(\mathbb{R}^4)}\\
&=\left|\sin\left(\frac{\UVWX+\XWVU}{2}\right)\right|^{\frac{2}{q}-1}\left|\sin\left(\frac{\PQRS+\SRQP}{2}\right)\right|^{\frac{2}{q}-1}A_p^4\|\tilde{f}\|_{L^p_\mathbb{H}(\mathbb{R}^4)}.
\end{align*}
By virtue of  \eqref{P7eqn15} we obtain \eqref{P7eqn18}. This finishes the proof.
\end{proof}
\begin{remark}
\begin{itemize}
\item For $\boldsymbol\alpha=\boldsymbol\beta,$ equation \eqref{P7eqn18} reduces to the sharp Hausdorff-Young inequality for the QFrFT  of the function $f\in L^2_{\mathbb{H}}(\mathbb{R}^4).$
\item For $\boldsymbol\alpha=\boldsymbol\beta=(\frac{\pi}{2},\frac{\pi}{2}),$ equation \eqref{P7eqn18} reduces to the sharp Hausdorff-Young inequality for the QFT  of the function $f\in L^2_{\mathbb{H}}(\mathbb{R}^4).$
\end{itemize}
\end{remark}
In what follows we obtain the Parseval's formula associated with  the QCFrFT.
\begin{theorem}\label{P7Theo3.2}
If $f,g\in L^2_\mathbb{H}(\mathbb{R}^4),$ then 
\begin{align}\label{P7eqn19}
\langle \mathcal{F}^{\boldsymbol\alpha,\boldsymbol\beta}_{\mathbb{H}} f, \mathcal{F}^{\boldsymbol\alpha,\boldsymbol\beta}_{\mathbb{H}} g \rangle=\langle f,g\rangle.
\end{align}
In particular,
\begin{align}\label{P7eqn20}
\|f\|^2_{L^2_\mathbb{H}(\mathbb{R}^4)}=\|\mathcal{F}^{\boldsymbol\alpha,\boldsymbol\beta}_{\mathbb{H}} f\|^2_{L^2_\mathbb{H}(\mathbb{R}^4)}.
\end{align}
\end{theorem}
\begin{proof}
Applying Parseval's formula for the QFT, we get
\begin{align*}
\langle \tilde{f},\tilde{g}\rangle&=\langle \mathcal{F}_{\mathbb{H}}\tilde{f},\mathcal{F}_{\mathbb{H}}\tilde{g}\rangle\\
&=Sc\left[|\det(-M_1)||\det(-M_2)|\int_{\mathbb{R}^4}\left(\mathcal{F}_{\mathbb{H}}\tilde{f}\right)(-M_1\bldxi_1,-M_2\bldxi_2)\overline{\left(\mathcal{F}_{\mathbb{H}}\tilde{g}\right)(-M_1\bldxi_1,-M_2\bldxi_2)}d\bldxi\right].
\end{align*}
Using relation \eqref{P7eqn16}, we have
\begin{align*}
\langle \tilde{f},\tilde{g}\rangle&=\frac{|\det(-M_1)||\det(-M_2)|}{|\tilde{d}_0(\MNOP)|^2}\int_{\mathbb{R}^4}Sc\left[\frac{1}{\tilde{d}_0(\PONM)}e^{i\tilde{a}(\PONM)|\bldxi_1|^2}\left(\mathcal{F}^{\boldsymbol\alpha,\boldsymbol\beta}_{\mathbb{H}} f \right)(\bldxi)\overline{\left(\mathcal{F}^{\boldsymbol\alpha,\boldsymbol\beta}_{\mathbb{H}} g\right)(\bldxi)}\frac{1}{\overline{\tilde{d}_0(\PONM)}}e^{-i\tilde{a}(\PONM)|\bldxi_1|^2}\right]d\bldxi\\
&=\frac{|\det(-M_1)||\det(-M_2)|}{|\tilde{d}_0(\MNOP)|^2|\tilde{d}_0(\PONM)|^2}\int_{\mathbb{R}^4}Sc\left[\left(\mathcal{F}^{\boldsymbol\alpha,\boldsymbol\beta}_{\mathbb{H}} f \right)(\bldxi)\overline{\left(\mathcal{F}^{\boldsymbol\alpha,\boldsymbol\beta}_{\mathbb{H}} g\right)(\bldxi)}\right]d\bldxi\\
&=\langle \mathcal{F}^{\boldsymbol\alpha,\boldsymbol\beta}_{\mathbb{H}} f, \mathcal{F}^{\boldsymbol\alpha,\boldsymbol\beta}_{\mathbb{H}} g \rangle.
\end{align*}
Again, using equation \eqref{P7eqn15}, it can be shown that $\langle \tilde{f}, \tilde{g}\rangle=\langle f, g\rangle.$ Hence \eqref{P7eqn19} follows. With $f=g,$ in \eqref{P7eqn19}, we have \eqref{P7eqn20}.
 
This finishes the proof.
\end{proof}

\subsection{R\`enyi entropy uncertainty principle}
The R\`enyi entropy UPs for the proposed QCFrFT is obtained in this subsection. Similar findings for the complex FrFT can be seen in \cite{guanlei2009generalized}. These UPs for the QPFT and two sided quaternion QPFT have recently been discovered in \cite{shah2021uncertainty} and \cite{gupta2022short},\cite{bhat2022uncertainty} respectively. We recall the following.
\begin{definition}\label{P7Defn3.2}
\cite{dembo1991information, guanlei2009generalized}
If $P$ is a probability density function on $\mathbb{R}^n,$ then the R\`enyi entropy of $P$ is defined by
\begin{align}\label{P7eqn21}
H_s(P)=\frac{1}{1-s}\log \left(\int_{\mathbb{R}^n}[P(\bldt)]^s d\bldt\right),~s>0, s\neq 1.
\end{align}
If $s\rightarrow 1,$ then \eqref{P7eqn21} results in the Shannon entropy given by
\begin{align}\label{P7eqn22}
E(P)=-\int_{\mathbb{R}^n}P(\bldt)\log [P(\bldt)]d\bldt.
\end{align}
\end{definition}
In what follows, we obtain the R\`enyi entropy UP for QCFrFT. 
\begin{theorem}\label{P7Theo3.3}
If $f\in L^2_\mathbb{H}(\mathbb{R}^4),$ $\frac{1}{2}<\alpha_0<1$ and $\frac{1}{\alpha_0}+\frac{1}{\beta_0}=2,$ then 
\begin{align*}
H_{\alpha_0}(|f|^2)+H_{\beta_0}\left(\left|\left(\mathcal{F}^{\boldsymbol\alpha,\boldsymbol\beta}_\mathbb{H}f\right)(\bldxi)\right|^{2}\right)\geq\frac{2}{\alpha_0-1}\log(2\alpha_0)+\frac{2}{\beta_0-1}\log(2\beta_0)+2\log\left(\left|\sin\left(\frac{\UVWX+\XWVU}{2}\right)\right|\left|\sin\left(\frac{\PQRS+\SRQP}{2}\right)\right|\right).
\end{align*}
\end{theorem}
\begin{proof}
By inequality \eqref{P7eqn18}, we have 
\begin{align}\label{P7eqn23}
\left(\int_{\mathbb{R}^4}\left|\left(\mathcal{F}^{\boldsymbol\alpha,\boldsymbol\beta}_{\mathbb{H}}f\right)(\bldxi)\right|^qd\bldxi\right)^{\frac{1}{q}}\leq \left|\sin\left(\frac{\UVWX+\XWVU}{2}\right)\right|^{\frac{2}{q}-1}\left|\sin\left(\frac{\PQRS+\SRQP}{2}\right)\right|^{\frac{2}{q}-1}\left(\int_{\mathbb{R}^4}|f(\bldt)|^pd\bldt\right)^{\frac{1}{p}}.
\end{align} 
Putting $p=2\alpha_0$ and $q=2\beta_0,$ in equation \eqref{P7eqn23}, we have
\begin{align*}
\left(\int_{\mathbb{R}^4}\left|\left(\mathcal{F}^{\boldsymbol\alpha,\boldsymbol\beta}_{\mathbb{H}}f\right)(\bldxi)\right|^{2\beta_0}d\bldxi\right)^{\frac{1}{2\beta_0}}\leq \left|\sin\left(\frac{\UVWX+\XWVU}{2}\right)\right|^{\frac{1}{\beta_0}-1}\left|\sin\left(\frac{\PQRS+\SRQP}{2}\right)\right|^{\frac{1}{\beta_0}-1}\left(\int_{\mathbb{R}^4}|f(\bldt)|^{2\alpha_0}d\bldt\right)^{\frac{1}{2\alpha_0}}.
\end{align*}
This implies
\begin{align}\label{P7eqn24}
\frac{\left|\sin\left(\frac{\UVWX+\XWVU}{2}\right)\right|^{2-\frac{2}{\beta_0}}\left|\sin\left(\frac{\PQRS+\SRQP}{2}\right)\right|^{2-\frac{2}{\beta_0}}}{A_{2\alpha_0}^8}\leq \left(\int_{\mathbb{R}^4}|f(\bldt)|^{2\alpha_0}d\bldt\right)^{\frac{1}{\alpha_0}}\left(\int_{\mathbb{R}^4}\left|\left(\mathcal{F}^{\boldsymbol\alpha,\boldsymbol\beta}_{\mathbb{H}}f\right)(\bldxi)\right|^{2\beta_0}d\bldxi\right)^{-\frac{1}{\beta_0}}.
\end{align}
Since $\frac{1}{\alpha_0}+\frac{1}{\beta_0}=2,$ we have
\begin{align}\label{P7eqn25}
\frac{\alpha_0}{1-\alpha_0}=\frac{\beta_0}{1-\beta_0}.
\end{align}
Raising to the power $\frac{\alpha_0}{1-\alpha_0}$ in \eqref{P7eqn24} and using \eqref{P7eqn25}, we get
\begin{align*}
\frac{\left|\sin\left(\frac{\UVWX+\XWVU}{2}\right)\right|^2\left|\sin\left(\frac{\PQRS+\SRQP}{2}\right)\right|^2}{\left(A_{2\alpha_0}^8\right)^{\frac{\alpha_0}{1-\alpha_0}}}\leq \left(\int_{\mathbb{R}^4}|f(\bldt)|^{2\alpha_0}d\bldt\right)^{\frac{1}{1-\alpha_0}}\left(\int_{\mathbb{R}^4}\left|\left(\mathcal{F}^{\boldsymbol\alpha,\boldsymbol\beta}_{\mathbb{H}}f\right)(\bldxi)\right|^{2\beta_0}d\bldxi\right)^{\frac{1}{1-\beta_0}}.
\end{align*}
This leads to 
\begin{align*}
-2\log\left(\left|\sin\left(\frac{\UVWX+\XWVU}{2}\right)\right|\left|\sin\left(\frac{\PQRS+\SRQP}{2}\right)\right|\right)-&\frac{8\alpha_0}{1-\alpha_0}\log\left(A_{2\alpha_0}\right)\notag\\
&\leq \frac{1}{1-\alpha_0}\log \left(\int_{\mathbb{R}^4}|f(\bldt)|^{2\alpha_0}d\bldt\right)+\frac{1}{1-\beta_0}\log \left(\int_{\mathbb{R}^4}\left|\left(\mathcal{F}^{\boldsymbol\alpha,\boldsymbol\beta}_\mathbb{H}\right)(\bldxi)\right|^{2\beta_0}d\bldxi\right).
\end{align*}
This implies
\begin{align}\label{P7eqn27}
H_{\alpha_0}(|f|^2)+H_{\beta_0}\left(\left|\mathcal{F}^{\boldsymbol\alpha,\boldsymbol\beta}_\mathbb{H}f\right|^{2}\right)\geq\frac{2}{\alpha_0-1}\log(2\alpha_0)+\frac{2}{\beta_0-1}\log(2\beta_0)+2\log\left(\left|\sin\left(\frac{\UVWX+\XWVU}{2}\right)\right|\left|\sin\left(\frac{\PQRS+\SRQP}{2}\right)\right|\right).
\end{align}
Thus the proof is complete.
\end{proof}
\begin{remark}
If $\alpha_0\rightarrow 1,$ then $\beta_0\rightarrow 1$ and thus \eqref{P7eqn27} can be written as
\begin{align}\label{P7eqn28}
&E(|f|^2)+E\left(\left|\left(\mathcal{F}^{\boldsymbol\alpha,\boldsymbol\beta}_\mathbb{H}f\right)(\bldxi)\right|^{2}\right)\geq 2\log\left(\left|\sin\left(\frac{\UVWX+\XWVU}{2}\right)\right|\left|\sin\left(\frac{\PQRS+\SRQP}{2}\right)\right|\right)+2(2-\log4)\notag,\\
&\hspace{-0.6cm}\mbox{i.e.,}~ E(|f|^2)+E\left(\left|\left(\mathcal{F}^{\boldsymbol\alpha,\boldsymbol\beta}_\mathbb{H}f\right)(\bldxi)\right|^{2}\right)\geq 2\log\left(\frac{e^2}{4}\left|\sin\left(\frac{\UVWX+\XWVU}{2}\right)\right|\left|\sin\left(\frac{\PQRS+\SRQP}{2}\right)\right|\right).
\end{align}
For QCFrFT, inequality \eqref{P7eqn28} is the Shannon entropy UP.
\end{remark}

\section{Short time quaternion coupled fractional Fourier transform}
The two sided short time quaternion coupled fractional Fourier transform (STQCFrFT) is defined and its properties are examined in this section .
\begin{definition}\label{P7Defn4.1}
Let $\boldsymbol\alpha=(\UVWX,\PQRS),\boldsymbol\beta=(\XWVU,\SRQP) \in\mathbb{R}^2$ such that $\UVWX+\XWVU,~\PQRS+\SRQP\notin 2\pi\mathbb{Z} $. The STQCFrFT of a function $f\in L^2_\mathbb{H}(\mathbb{R}^4)$ with respect to $g\in L^2_\mathbb{H}(\mathbb{R}^4)\cap L^\infty_\mathbb{H}(\mathbb{R}^4),$ called a quaternion window function (QWF), is defined by
\begin{align}\label{P7Definition_STQCFrFT}
\left(\mathcal{S}^{\boldsymbol\alpha,\boldsymbol\beta}_{\mathbb{H},g}f\right)(\bldx,\bldxi)=\int_{\mathbb{R}^4}\mathcal{K}^i_{\UVWX,\XWVU}(\bldt_1,\bldxi_1)f(\bldt)\overline{g(\bldt-\bldx)}\mathcal{K}^j_{\PQRS,\SRQP}(\bldt_2,\bldxi_2)d\bldt,~(\bldx,\bldxi)\in\mathbb{R}^4\times\mathbb{R}^4,
\end{align}
where $\mathcal{K}^i_{\UVWX,\XWVU}(\bldt_1,\bldxi_1)$ and $\mathcal{K}^j_{\PQRS,\SRQP}(\bldt_2,\bldxi_2)$ are given respectively by \eqref{P7eqn12} and \eqref{P7eqn13}.
\end{definition}
It is to be noted that if the kernels $\mathcal{K}^i_{\UVWX,\XWVU}(\bldt_1,\bldxi_1)$ and $\mathcal{K}^j_{\PQRS,\SRQP}(\bldt_2,\bldxi_2)$ both lies together to the left or to the right of $f(\bldt)\overline{g(\bldt-\bldx)}$ in the integral in \eqref{P7Definition_STQCFrFT}, then we have the left or right sided STQCFrFT. But our study mainly focus on the two sided STQCFrFT.

\begin{remark} For $\boldsymbol\alpha=\boldsymbol\beta,$ the kernels $\mathcal{K}^i_{\UVWX,\UVWX}(\bldt_1,\bldxi_1)$ and $\mathcal{K}^j_{\PQRS,\PQRS}(\bldt_2,\bldxi_2)$ are the tensor product of two one-dimensional FrFT. Thus, the  STQCFrFT reduces to the STQFrFT of the function $f\in L^2_{\mathbb{H}}(\mathbb{R}^4).$ Moreover, if $\boldsymbol\alpha=\boldsymbol\beta=(\frac{\pi}{2},\frac{\pi}{2}),$ the STQCFrFT reduces to the STQFT of the function $f\in L^2_{\mathbb{H}}(\mathbb{R}^4).$
\end{remark}
We mention below a lemma that will be used in proving some basic properties of the STQCFrFT.
\begin{lemma}\label{P7Lemma4.1}
Let $\boldsymbol k=(\boldsymbol k_1,\boldsymbol k_2),\bldxi=(\bldxi_1,\bldxi_2),\bldt=(\bldt_1,\bldt_2)\in\mathbb{R}^2\times\mathbb{R}^2.$ Then $\mathcal{K}^i_{\UVWX,\XWVU}(\bldt_1,\bldxi_1)$ and $\mathcal{K}^j_{\PQRS,\SRQP}(\bldt_2,\bldxi_2)$ satisfy the following
\begin{align}\label{P7eqn29}
\mathcal{K}^i_{\UVWX,\XWVU}(\bldt_1+\boldsymbol k_1,\bldxi_1)=e^{-i\left\{\tilde{a}(\PONM)\left(|\boldsymbol k_1|+2\bldt_1\cdot\boldsymbol k_1\right)-\boldsymbol k_1\cdot M_1\bldxi_1\right\}}\mathcal{K}^i_{\UVWX,\XWVU}(\bldt_1,\bldxi_1),
\end{align}
and 
\begin{align}\label{P7eqn31}
\mathcal{K}^j_{\PQRS,\SRQP}(\bldt_2+\boldsymbol k_2,\bldxi_2)=e^{-j\left\{\tilde{a}(\MNOP)\left(|\boldsymbol k_2|+2\bldt_2\cdot\boldsymbol k_2\right)-\boldsymbol k_2\cdot M_2\bldxi_2\right\}}\mathcal{K}^j_{\PQRS,\SRQP}(\bldt_2,\bldxi_2).
\end{align}
\end{lemma}
\begin{proof}
From the definition of $\mathcal{K}^i_{\UVWX,\XWVU},$ we have
\begin{align*}
\mathcal{K}^i_{\UVWX,\XWVU}(\bldt_1+\boldsymbol k_1,\bldxi_1)
&=\tilde{d}(\PONM)e^{-i\left\{\tilde{a}(\PONM)\left(|\bldt_1+\boldsymbol k_1|^2+|\bldxi_1|^2\right)-(\bldt_1+\boldsymbol k_1)\cdot M_1\bldxi_1\right\}}\\
&=\tilde{d}(\PONM)e^{-i\left\{\tilde{a}(\PONM)\left(|\bldt_1|^2+|\boldsymbol k_1|^2+2\bldt_1\cdot\boldsymbol k_1+|\bldxi_1|^2\right)-\bldt_1\cdot M_1\bldxi_1-\boldsymbol k_1\cdot M_1\bldxi_1\right\}}\\
&=\tilde{d}(\PONM)e^{-i\left\{\tilde{a}(\PONM)\left(|\bldt_1|^2+|\bldxi_1|^2\right)-\bldt_1\cdot M_1\bldxi_1\right\}}e^{-i\left\{\tilde{a}(\PONM)\left(|\boldsymbol k_1|^2+2\bldt_1\cdot\bldxi_1\right)-\boldsymbol k_1\cdot M_1\bldxi_1\right\}},\\
\mbox{i.e.,}~\mathcal{K}^i_{\UVWX,\XWVU}(\bldt_1+\boldsymbol k_1,\bldxi_1)&=e^{-i\left\{\tilde{a}(\PONM)\left(|\boldsymbol k_1|+2\bldt_1\cdot\boldsymbol k_1\right)-\boldsymbol k_1\cdot M_1\bldxi_1\right\}}\mathcal{K}^i_{\UVWX,\XWVU}(\bldt_1,\bldxi_1).
\end{align*}
This proves \eqref{P7eqn29}. \eqref{P7eqn31} follows similarly.
\end{proof}
The proposed STQCFrFT enjoys the following basic properties.
\begin{theorem}\label{P7Theo4.2}
Let $f_1,f_2,f \in L^2_\mathbb{H}(\mathbb{R}^4)$ and $g_1,g,g_2\in L^\infty_\mathbb{H}(\mathbb{R}^4)\cap L^2_\mathbb{H}(\mathbb{R}^4)$ be QWFs. Then
\begin{enumerate}[label=(\roman*)]
\item{Boundedness}: $\left\|\mathcal{S}^{\boldsymbol\alpha,\boldsymbol\beta}_{\mathbb{H},g}f\right\|_{L^\infty_\mathbb{H}(\mathbb{R}^4)}\leq\frac{1}{4\pi^2\left|\sin\left(\frac{\UVWX+\XWVU}{2}\right)\right| \left|\sin\left(\frac{\PQRS+\SRQP}{2}\right)\right|}\|g\|_{L^2_\mathbb{H}(\mathbb{R}^4)}\|f\|_{L^2_\mathbb{H}(\mathbb{R}^4)}.$ 
\item {Linearity}:
$\mathcal{S}^{\boldsymbol\alpha,\boldsymbol\beta}_{\mathbb{H},g}(pf_1+qf_2)=p\left[\mathcal{S}^{\boldsymbol\alpha,\boldsymbol\beta}_{\mathbb{H},g}f_1\right]+q\left[\mathcal{S}^{\boldsymbol\alpha,\boldsymbol\beta}_{\mathbb{H},g}f_2\right],~p,q\in\{\xx+i\yy:\xx,\yy\in\mathbb{R}\}$
\item {Anti-linearity}:
$\mathcal{S}^{\boldsymbol\alpha,\boldsymbol\beta}_{\mathbb{H},rg_1+sg_2}f=\left[\mathcal{S}^{\boldsymbol\alpha,\boldsymbol\beta}_{\mathbb{H},g_1}f\right]\bar{r}+\left[\mathcal{S}^{\boldsymbol\alpha,\boldsymbol\beta}_{\mathbb{H},g_2}f\right]\bar{s},~r,s\in\{\xx+j\yy:\xx,\yy\in\mathbb{R}\}.$
\item{Translation}: $\left(\mathcal{S}^{\boldsymbol\alpha,\boldsymbol\beta}_{\mathbb{H},g}(\tau_{\boldsymbol l}f)\right)(\bldx,\bldxi)=e^{-i\left\{\tilde{a}(\PONM)\left(|\boldsymbol l_1|+2\bldt_1\cdot\boldsymbol l_1\right)-\boldsymbol l_1\cdot M_1\bldxi_1\right\}}\left(\mathcal{S}^{\boldsymbol\alpha,\boldsymbol\beta}_{\mathbb{H},g}f\right)(\bldx-\boldsymbol l,\bldxi)e^{-j\left\{\tilde{a}(\MNOP)\left(|\boldsymbol l_2|+2\bldt_2\cdot\boldsymbol l_2\right)-\boldsymbol l_2\cdot M_2\bldxi_2\right\}},$ where 
$(\tau_{\boldsymbol l}f)(\bldt)=f(\bldt-\boldsymbol l),~\boldsymbol l=(\boldsymbol l_1,\boldsymbol l_2)\in\mathbb{R}^2\times\mathbb{R}^2.$ 
 
\item {Parity}: $\left(\mathcal{S}^{\boldsymbol\alpha,\boldsymbol\beta}_{\mathbb{H},Pg}Pf\right)(\bldx,\bldxi)=\left(\mathcal{S}^{\boldsymbol\alpha,\boldsymbol\beta}_{\mathbb{H},g}f\right)\left(-\bldx,-\bldxi\right),$ where $(Pf)(\bldt)=f\left(-\bldt\right).$
\end{enumerate}
\end{theorem}
\begin{proof}
We skip the proof of $(i),$ $(ii)$ and $(iii)$ as they are straight forward.

$(iv)$ From definition \ref{P7Defn4.1}, we have
\begin{align*}
\left(\mathcal{S}^{\boldsymbol\alpha,\boldsymbol\beta}_{\mathbb{H},g}(\tau_{\boldsymbol l}f)\right)(\bldx,\bldxi)=\int_{\mathbb{R}^4}\mathcal{K}^i_{\UVWX,\XWVU}(\bldt_1+\boldsymbol l_1,\bldxi_1)f(\bldt)\overline{g(\bldt-(\bldx-\boldsymbol l))}\mathcal{K}^j_{\PQRS,\SRQP}(\bldt_2+\boldsymbol l_2,\bldxi_2)d\bldt.
\end{align*}
By lemma \eqref{P7Lemma4.1}, we see that
\begin{align*}
&\left(\mathcal{S}^{\boldsymbol\alpha,\boldsymbol\beta}_{\mathbb{H},g}(\tau_{\boldsymbol l}f)\right)(\bldx,\bldxi)\\
&=e^{-i\left\{\tilde{a}(\PONM)\left(|\boldsymbol l_1|+2\bldt_1\cdot\boldsymbol l_1\right)-\boldsymbol l_1\cdot M_1\bldxi_1\right\}}\left\{\int_{\mathbb{R}^4}\mathcal{K}^i_{\UVWX,\XWVU}(\bldt_1,\bldxi_1)f(\bldt)\overline{g(\bldt-(\bldx-\boldsymbol l))}\mathcal{K}^j_{\UVWX,\XWVU}(\bldt_2,\bldxi_2)d\bldt\right\}e^{-j\left\{\tilde{a}(\MNOP)\left(|\boldsymbol l_2|+2\bldt_2\cdot\boldsymbol l_2\right)-\boldsymbol l_2\cdot M_2\bldxi_2\right\}}.
\end{align*}
Thus, we have
\begin{align*}
\left(\mathcal{S}^{\boldsymbol\alpha,\boldsymbol\beta}_{\mathbb{H},g}(\tau_{\boldsymbol l}f)\right)(\bldx,\bldxi)=e^{-i\left\{\tilde{a}(\PONM)\left(|\boldsymbol l_1|+2\bldt_1\cdot\boldsymbol l_1\right)-\boldsymbol l_1\cdot M_1\bldxi_1\right\}}\left(\mathcal{S}^{\boldsymbol\alpha,\boldsymbol\beta}_{\mathbb{H},g}f\right)(\bldx-\boldsymbol l,\bldxi)e^{-j\left\{\tilde{a}(\MNOP)\left(|\boldsymbol l_2|+2\bldt_2\cdot\boldsymbol l_2\right)-\boldsymbol l_2\cdot M_2\bldxi_2\right\}}.
\end{align*}
This proves $(iii).$

$(v)$ Using the definition of STQCFrFT, we have
\begin{align}\label{P7eqn33}
\left(\mathcal{S}^{\wedge_1,\wedge_2}_{\mathbb{H},Pg}Pf\right)(\bldx,\bldxi)
&=\int_{\mathbb{R}^4}\mathcal{K}^i_{\UVWX,\XWVU}(-\bldt_1,\bldxi_1)f(\bldt)\overline{g\left(\bldt-\bldx\right)}\mathcal{K}^j_{\PQRS,\SRQP}(-\bldt_2,\bldxi_2)d\bldt.
\end{align}
Now, it can be shown that 
\begin{align}\label{P7eqn34}
\mathcal{K}^i_{\UVWX,\XWVU}(-\bldt_1,\bldxi_1)=\mathcal{K}^i_{\UVWX,\XWVU}(\bldt_1,-\bldxi_1)
\end{align}
and
\begin{align}\label{P7eqn35}
\mathcal{K}^j_{\PQRS,\SRQP}(-\bldt_2,\bldxi_2)=\mathcal{K}^j_{\PQRS,\SRQP}(\bldt_2,-\bldxi_2).
\end{align}
By virtue of  \eqref{P7eqn34}, \eqref{P7eqn35} and \eqref{P7eqn33}, we have
\begin{align*}
\left(\mathcal{S}^{\boldsymbol\alpha,\boldsymbol\beta}_{\mathbb{H},Pg}Pf\right)(\bldx,\bldxi)=\left(\mathcal{S}^{\boldsymbol\alpha,\boldsymbol\beta}_{\mathbb{H},g}f\right)\left(-\bldx,-\bldxi\right).
\end{align*}
Thus the proof is complete.
\end{proof}
\begin{theorem}\label{P7Theo4.3}
(Inner product relation): If $f,h\in L^2_\mathbb{H}(\mathbb{R}^4)$ and $g_1,g_2$ are two QWFs, then $\mathcal{S}^{\boldsymbol\alpha,\boldsymbol\beta}_{\mathbb{H},g_1}f,~\mathcal{S}^{\boldsymbol\alpha,\boldsymbol\beta}_{\mathbb{H},g_2}f\in L^2_\mathbb{H}(\mathbb{R}^4\times\mathbb{R}^4)$ and
\begin{align}\label{P7eqn36}
\left\langle\mathcal{S}^{\boldsymbol\alpha,\boldsymbol\beta}_{\mathbb{H},g_1}f,\mathcal{S}^{\boldsymbol\alpha,\boldsymbol\beta}_{\mathbb{H},g_2}h\right\rangle=\langle f(\overline{g_1},\overline{g_2}),h\rangle.
\end{align}
\end{theorem}
\begin{proof}
We have
\begin{align*}
\int_{\mathbb{R}^4}\int_{\mathbb{R}^4}\left|\left(\mathcal{S}^{\boldsymbol\alpha,\boldsymbol\beta}_{\mathbb{H},g_1}f\right)(\bldx,\bldxi)\right|^2d\bldx d\bldxi
&=\int_{\mathbb{R}^4}\left\{\int_{\mathbb{R}^4}\left|\left(\mathcal{F}^{\boldsymbol\alpha,\boldsymbol\beta}_{\mathbb{H}}\{f(\cdot)\overline{g_1(\cdot-\bldx)}\}\right)(\bldxi)\right|^2d\bldxi\right\}d\bldx\\
&=\int_{\mathbb{R}^4}\left\{\int_{\mathbb{R}^4}|f(\bldt)\overline{g_1(\bldt-\bldx)}|^2d\bldt\right\}d\bldx,~\mbox{using Parseval's Identity}\\
&=\|f\|^2_{L^2_\mathbb{H}(\mathbb{R}^4)}\|g_1\|^2_{L^2_\mathbb{H}(\mathbb{R}^4)}.
\end{align*}
Thus, $\mathcal{S}^{\boldsymbol\alpha,\boldsymbol\beta}_{\mathbb{H},g_1}f\in L^2_\mathbb{H}(\mathbb{R}^4\times\mathbb{R}^4).$ Similarly, $\mathcal{S}^{\boldsymbol\alpha,\boldsymbol\beta}_{\mathbb{H},g_2}h\in L^2_\mathbb{H}(\mathbb{R}^4\times\mathbb{R}^4).$

Now,
\begin{align*}
\left\langle\mathcal{S}^{\boldsymbol\alpha,\boldsymbol\beta}_{\mathbb{H},g_1}f,\mathcal{S}^{\boldsymbol\alpha,\boldsymbol\beta}_{\mathbb{H},g_2}h\right\rangle
&=Sc\int_{\mathbb{R}^4}\int_{\mathbb{R}^4}\left(\mathcal{F}^{\boldsymbol\alpha,\boldsymbol\beta}_{\mathbb{H}}\{f(\cdot)\overline{g_1(\cdot-\bldx)}\}\right)(\bldxi)\overline{\left(\mathcal{F}^{\boldsymbol\alpha,\boldsymbol\beta}_{\mathbb{H}}\{h(\cdot)\overline{g_2(\cdot-\bldx)}\}\right)(\bldxi)}d\bldx d\bldxi\\
&=Sc\int_{\mathbb{R}^4}\left\{\int_{\mathbb{R}^4}f(\bldt)\overline{g_1(\bldt-\bldx)}~\overline{h(\bldt)\overline{g_{2}(\bldt-\bldx)}}d\bldt\right\}d\bldx\\
&=Sc\int_{\mathbb{R}^4}f(\bldt)
\left(\int_{\mathbb{R}^4}\overline{g_1(\bldt-\bldx)}g_2(\bldt-\bldx)d\bldx\right)
\overline{h(\bldt)}d\bldt\\
&=Sc\int_{\mathbb{R}^4}f(\bldt)\left(\overline{g_1},\overline{g_2}\right)\overline{h(\bldt)}d\bldt.
\end{align*}
This implies
\begin{align*}
\left\langle\mathcal{S}^{\boldsymbol\alpha,\boldsymbol\beta}_{\mathbb{H},g_1}f,\mathcal{S}^{\boldsymbol\alpha,\boldsymbol\beta}_{\mathbb{H},g_2}h\right\rangle=\langle f\left(\overline{g_1},\overline{g_2}\right),h\rangle.
\end{align*}
This finishes that proof.
\end{proof}
\begin{remark}
In view of \eqref{P7eqn36}, we can conclude that
\begin{enumerate}
\item If $g_1=g=g_2$ in \eqref{P7eqn36}, then 
\begin{align*}
\left\langle\mathcal{S}^{\boldsymbol\alpha,\boldsymbol\beta}_{\mathbb{H},g}f,\mathcal{S}^{\boldsymbol\alpha,\boldsymbol\beta}_{\mathbb{H},g}h\right\rangle=\langle f,h\rangle\|g\|^2_{L^2_\mathbb{H}(\mathbb{R}^4)}.
\end{align*}
\item If $f=h$ in \eqref{P7eqn36}, then
\begin{align*}
\left\langle\mathcal{S}^{\boldsymbol\alpha,\boldsymbol\beta}_{\mathbb{H},g_1}f,\mathcal{S}^{\boldsymbol\alpha,\boldsymbol\beta}_{\mathbb{H},g_2}f\right\rangle=\langle g_1,g_2\rangle\|f\|^2_{L^2_\mathbb{H}(\mathbb{R}^4)}.
\end{align*}
\item If  $g=g_1=g_2$ and $f=h$ in \eqref{P7eqn36}, then 
\begin{align}\label{P7eqn37}
\|\mathcal{S}^{\boldsymbol\alpha,\boldsymbol\beta}_{\mathbb{H},g}f\|^2_{L^2_\mathbb{H}(\mathbb{R}^4\times\mathbb{R}^4)}=\|g\|^2_{L^2_\mathbb{H}(\mathbb{R}^4)}\|f\|^2_{L^2_\mathbb{H}(\mathbb{R}^4)}.
\end{align}
\end{enumerate} 
\end{remark}
The following theorem gives the reconstruction formula for the STQCFrFT.
\begin{theorem}\label{P7Theo4.4}
If $g$ be a QWF and $f\in L^2_\mathbb{H}(\mathbb{R}^4)$, then
$$f(\bldt)=\int_{\mathbb{R}^4}\int_{\mathbb{R}^4}\overline{\mathcal{K}^i_{\UVWX,\XWVU}(\bldt_1,\bldxi_1)}\left(\mathcal{S}^{\boldsymbol\alpha,\boldsymbol\beta}_{\mathbb{H},g}f\right)(\bldx,\bldxi)\overline{\mathcal{K}^j_{\PQRS,\SRQP}(\bldt_2,\bldxi_2)}g(\bldt-\bldx)d\bldx d\bldxi.$$
\end{theorem}
\begin{proof}
Using theorem \ref{P7Theo4.3}, we see that
\begin{align*}
\langle f,h\rangle
&=Sc\int_{\mathbb{R}^4}\int_{\mathbb{R}^4}\left(\mathcal{S}^{\boldsymbol\alpha,\boldsymbol\beta}_{\mathbb{H},g}f\right)(\bldx,\bldxi)\overline{\left\{\int_{\mathbb{R}^4}\mathcal{K}^i_{\UVWX,\XWVU}(\bldt_1,\bldxi_1)h(\bldt)\overline{g(\bldt-\bldx)}\mathcal{K}^j_{\PQRS,\SRQP}(\bldt_2,\bldxi_2)d\bldt\right\}}d\bldx d\bldxi\\
&=\int_{\mathbb{R}^4}\int_{\mathbb{R}^4}\int_{\mathbb{R}^4}Sc\left\{\left(\mathcal{S}^{\boldsymbol\alpha,\boldsymbol\beta}_{\mathbb{H},g}f\right)(\bldx,\bldxi)\overline{\mathcal{K}^j_{\PQRS,\SRQP}(\bldt_2,\bldxi_2)}g(\bldt-\bldx)\overline{h(\bldt)}\overline{\mathcal{K}^i_{\UVWX,\XWVU}(\bldt_1,\bldxi_1)}\right\}d\bldt d\bldx d\bldxi\\
&=Sc\int_{\mathbb{R}^4}\left\{\int_{\mathbb{R}^4}\int_{\mathbb{R}^4} \overline{\mathcal{K}^i_{\UVWX,\XWVU}(\bldt_1,\bldxi_1)}\left(\mathcal{S}^{\boldsymbol\alpha,\boldsymbol\beta}_{\mathbb{H},g}f\right)(\bldx,\bldxi)\overline{\mathcal{K}^j_{\PQRS,\SRQP}(\bldt_2,\bldxi_2)}g(\bldt-\bldx)d\bldx d\bldxi\right\}\overline{h(\bldt)}d\bldt\\
&=\left\langle\int_{\mathbb{R}^4}\int_{\mathbb{R}^4} \overline{\mathcal{K}^i_{\UVWX,\XWVU}(\bldt_1,\bldxi_1)}\left(\mathcal{S}^{\boldsymbol\alpha,\boldsymbol\beta}_{\mathbb{H},g}f\right)(\bldx,\bldxi)\overline{\mathcal{K}^j_{\PQRS,\SRQP}(\bldt_2,\bldxi_2)}g(\cdot-\bldx)d\bldx d\bldxi,h(\cdot)\right\rangle.
\end{align*}
Since $h\in L^2_\mathbb{H}(\mathbb{R}^4)$ is arbitrary, it follows that
$$f(\bldt)=\int_{\mathbb{R}^4}\int_{\mathbb{R}^4}\overline{\mathcal{K}^i_{\UVWX,\XWVU}(\bldt_1,\bldxi_1)}\left(\mathcal{S}^{\boldsymbol\alpha,\boldsymbol\beta}_{\mathbb{H},g}f\right)(\bldx,\bldxi)\overline{\mathcal{K}^j_{\PQRS,\SRQP}(\bldt_2,\bldxi_2)}g(\bldt-\bldx)d\bldx d\bldxi.$$
This finishes the proof.
\end{proof}

\subsection{Uncertainty principle for STQCFrFT}
Similar to the Heisenberg UP, that governs the localization of a function and its FT, Wilczok (\cite{wilczok2000new}) presented a novel form of UP that compares the localization of both a function and its windowed FT. 

Here, we give the Lieb's UP for the STQCFrFT. In our recent paper \cite{gupta2022short} we have obtained an analogous result for the short time quaternion QPFT. We first obtain the Lieb's inequality for the STQCFrFT. 

\begin{lemma}\label{P7Lemma4.7}
Let $2\leq q<\infty,$ $f\in L^2_\mathbb{H}(\mathbb{R}^4)$ and $g$ be a QWF. Then
\begin{align}\label{P7eqn41}
\left\|\mathcal{S}^{\boldsymbol\alpha,\boldsymbol\beta}_{\mathbb{H},g}f \right\|_{L^q_\mathbb{H}(\mathbb{R}^4\times\mathbb{R}^4)}\leq\left|\sin\left(\frac{\UVWX+\XWVU}{2}\right)\right|^{\frac{2}{q}-1}\left|\sin\left(\frac{\PQRS+\SRQP}{2}\right)\right|^{\frac{2}{q}-1}\left(\frac{2}{q}\right)^\frac{4}{q}\|g\|_{L^2_\mathbb{H}(\mathbb{R}^4)}\|f\|_{L^2_\mathbb{H}(\mathbb{R}^4)}.
\end{align}
\end{lemma}
\begin{proof}
Using the definition \ref{P7Defn4.1}, we have
\begin{align}\label{P7eqn42}
\left(\int_{\mathbb{R}^4}\left|\left(\mathcal{S}^{\boldsymbol\alpha,\boldsymbol\beta}_{\mathbb{H},g}f \right)(\bldx,\bldxi)\right|^qd\bldxi\right)^\frac{1}{q}=\left(\int_{\mathbb{R}^4}\left|\left(\mathcal{F}^{\boldsymbol\alpha,\boldsymbol\beta}_{\mathbb{H}}\{f(\cdot)\overline{g(\cdot-\bldx)}\}\right)(\bldxi)\right|^qd\bldxi\right)^\frac{1}{q}.
\end{align}
Using Hausdorff-Young inequality, we get
\begin{align*}
\left(\int_{\mathbb{R}^4}\left|\left(\mathcal{S}^{\boldsymbol\alpha,\boldsymbol\beta}_{\mathbb{H},g}f \right)(\bldx,\bldxi)\right|^qd\bldxi\right)^\frac{1}{q}
&\leq A_p^4\left|\sin\left(\frac{\UVWX+\XWVU}{2}\right)\right|^{\frac{2}{q}-1}\left|\sin\left(\frac{\PQRS+\SRQP}{2}\right)\right|^{\frac{2}{q}-1}\left(\int_{\mathbb{R}^4}\left|f(\bldt)\overline{g(\bldt-\bldx)}\right|^pd\bldt\right)^{\frac{1}{p}}\\
&=A_p^4\left|\sin\left(\frac{\UVWX+\XWVU}{2}\right)\right|^{\frac{2}{q}-1}\left|\sin\left(\frac{\PQRS+\SRQP}{2}\right)\right|^{\frac{2}{q}-1}\left(\int_{\mathbb{R}^4}|f(\bldt)|^p|\tilde{g}(\bldx-\bldt)|^pd\bldt\right)^{\frac{1}{p}},~\tilde{g}(\bldt)=g(-\bldt)\\
&=A_p^4\left|\sin\left(\frac{\UVWX+\XWVU}{2}\right)\right|^{\frac{2}{q}-1}\left|\sin\left(\frac{\PQRS+\SRQP}{2}\right)\right|^{\frac{2}{q}-1}\left\{\left(|f|^p\star|\tilde{g}|^p\right)(\bldx)\right\}^{\frac{1}{p}}.
\end{align*}
This leads to
\begin{align*}
\int_{\mathbb{R}^4}\int_{\mathbb{R}^4}\left|\left(\mathcal{S}^{\boldsymbol\alpha,\boldsymbol\beta}_{\mathbb{H},g}f \right)(\bldx,\bldxi)\right|^qd\bldx d\bldxi\leq A_p^{4q}\left|\sin\left(\frac{\UVWX+\XWVU}{2}\right)\right|^{2-q}\left|\sin\left(\frac{\PQRS+\SRQP}{2}\right)\right|^{2-q}\int_{\mathbb{R}^4}\left\{\left(|f|^p\star|\tilde{g}|^p\right)(\bldx)\right\}^{\frac{q}{p}}d\bldx.
\end{align*}
This implies
\begin{align}\label{P7eqn43}
\left\{\int_{\mathbb{R}^4}\int_{\mathbb{R}^4}\left|\left(\mathcal{S}^{\boldsymbol\alpha,\boldsymbol\beta}_{\mathbb{H},g}f \right)(\bldx,\bldxi)\right|^qd\bldx d\bldxi\right\}^\frac{1}{q}
&\leq A_p^{4q}\left|\sin\left(\frac{\UVWX+\XWVU}{2}\right)\right|^{2-q}\left|\sin\left(\frac{\PQRS+\SRQP}{2}\right)\right|^{2-q}\left[\int_{\mathbb{R}^4}\left\{\left(|f|^p\star|\tilde{g}|^p\right)(\bldx)\right\}^{\frac{q}{p}}d\bldx\right]^{\frac{p}{q}\cdot\frac{1}{p}}\notag\\
&=A_p^{4}\left|\sin\left(\frac{\UVWX+\XWVU}{2}\right)\right|^{\frac{2}{q}-1}\left|\sin\left(\frac{\PQRS+\SRQP}{2}\right)\right|^{\frac{2}{q}-1}\left\||f|^p\star|\tilde{g}|^p\right\|^{\frac{1}{p}}_{L^{\frac{q}{p}}_\mathbb{H}(\mathbb{R}^4)}.
\end{align}
We observe that, if $l=\frac{q}{p},~k=\frac{2}{p},$ then $k\geq 1$ and $1+\frac{1}{l}=\frac{1}{k}+\frac{1}{k}.$
Since $|\tilde{g}|^p,~|f|^p\in L^k_\mathbb{H}(\mathbb{R}^4),$ by Young's inequality, we have
\begin{align}\label{P7eqn44}
\left\||f|^p\star|\tilde{g}|^p\right\|^{\frac{1}{p}}_{L^{\frac{q}{p}}_\mathbb{H}(\mathbb{R}^4)}\leq A_k^8A_{l'}^4\|f\|^p_{L^2_\mathbb{H}(\mathbb{R}^4)}\|\tilde{g}\|^p_{L^2_\mathbb{H}(\mathbb{R}^4)}.
\end{align}
Thus, by virtue of \eqref{P7eqn43} and \eqref{P7eqn44}, it follows that
\begin{align}\label{P7eqn45}
\left\{\int_{\mathbb{R}^4}\int_{\mathbb{R}^4}\left|\left(\mathcal{S}^{\boldsymbol\alpha,\boldsymbol\beta}_{\mathbb{H},g}f \right)(\bldx,\bldxi)\right|^qd\bldx d\bldxi\right\}^\frac{1}{q}\leq\left|\sin\left(\frac{\UVWX+\XWVU}{2}\right)\right|^{\frac{2}{q}-1}\left|\sin\left(\frac{\PQRS+\SRQP}{2}\right)\right|^{\frac{2}{q}-1} \left(A_p^2A_k^{\frac{4}{p}}A_{l'}^{\frac{2}{p}}\right)^{2}\|g\|_{L^2_\mathbb{H}(\mathbb{R}^4)}\|f\|_{L^2_\mathbb{H}(\mathbb{R}^4)},
\end{align}
where $A_r=\left(\frac{r^{\frac{1}{r}}}{r'^{\frac{1}{r'}}}\right)^{\frac{1}{2}},~\frac{1}{r}+\frac{1}{r'}=1.$ 
Now, we have
\begin{align}\label{P7eqn46}
A_p^2A_k^{\frac{4}{p}}A_{l'}^{\frac{2}{p}}
&=\frac{p^{\frac{1}{p}}}{q\frac{1}{q}}\cdot\frac{k}{{k'}^\frac{2}{k'p}}\cdot\frac{{l'}^\frac{1}{pl'}}{\left(\frac{q}{p}\right)^{\frac{1}{q}}},~\mbox{since}~k=\frac{2}{q},~l=\frac{q}{p}\notag\\
&=\frac{p}{q^{\frac{2}{q}}}\cdot\frac{{l'}^{\frac{1}{pl'}}}{{k'}^{\frac{2}{k'p}}}\notag\\
&\leq \frac{2}{q^{\frac{2}{q}}}\cdot\left(\frac{1}{2}\right)^{\frac{q-p}{pq}},~\mbox{since}~k'=2l'\notag\\
&=\left(\frac{2}{q}\right)^{\frac{2}{q}}.
\end{align}
Thus using equation \eqref{P7eqn46} in \eqref{P7eqn45}, we get
\begin{align*}
\left\{\int_{\mathbb{R}^4}\int_{\mathbb{R}^4}\left|\left(\mathcal{S}^{\boldsymbol\alpha,\boldsymbol\beta}_{\mathbb{H},g}f \right)(\bldx,\bldxi)\right|^qd\bldx d\bldxi\right\}^\frac{1}{q}\leq \left|\sin\left(\frac{\UVWX+\XWVU}{2}\right)\right|^{\frac{2}{q}-1}\left|\sin\left(\frac{\PQRS+\SRQP}{2}\right)\right|^{\frac{2}{q}-1} \left(\frac{2}{q}\right)^{\frac{4}{q}}\|f\|_{L^2_\mathbb{H}(\mathbb{R}^4)}\|g\|_{L^2_\mathbb{H}(\mathbb{R}^4)}.
\end{align*}
This completes the proof.
\end{proof}
\begin{remark} For $\boldsymbol\alpha=\boldsymbol\beta,$ the  inequality \eqref{P7eqn41} reduces to the Lieb's inequality for the STQFrFT of the function $f\in L^2_{\mathbb{H}}(\mathbb{R}^4).$ Moreover, if $\boldsymbol\alpha=\boldsymbol\beta=(\frac{\pi}{2},\frac{\pi}{2}),$ inequality \eqref{P7eqn41} reduces to the Lieb's inequality for the STQFT of the function $f\in L^2_{\mathbb{H}}(\mathbb{R}^4).$
\end{remark}
\subsection{Lieb's uncertainty principle}
\begin{definition}\label{P7Defn4.4}
Let $\Omega\subset\mathbb{R}^n$ be measurable. A function $G\in L^2_\mathbb{H}(\mathbb{R}^n)$ is $\epsilon-$concentrated, $\epsilon\geq 0,$  on $\Omega$ if 
$$\|\chi_{\Omega^c}G\|_{L^2_\mathbb{H}(\mathbb{R}^n)}\leq \epsilon\|G\|_{L^2_\mathbb{H}(\mathbb{R}^n)},$$
where $\chi_{\Omega}$ takes the value $1$ on $\Omega$ and $0$ otherwise. 
\end{definition}
\begin{theorem}\label{P7Theo4.8}
Let $0\neq f\in L^2_\mathbb{H}(\mathbb{R}^4)$ and $g$ be a QWF. If $\epsilon\geq 0$ and $\mathcal{S}^{\boldsymbol\alpha,\boldsymbol\beta}_{\mathbb{H},g}f$ is $\epsilon-$concentrated on $\Omega\subset\mathbb{R}^4\times\mathbb{R}^4,$ then
\begin{align}\label{P7Bound_of_EssentialSupportofOmega}
|\Omega|\geq \left|\sin\left(\frac{\UVWX+\XWVU}{2}\right)\right|^2\left|\sin\left(\frac{\PQRS+\SRQP}{2}\right)\right|^2(1-\epsilon^2)^{\frac{q}{q-2}}\left(\frac{q}{2}\right)^{\frac{8}{q-2}},~q>2
\end{align}
where $|\Omega|$ denoted the Lebesgue measure of $\Omega.$
\end{theorem}
\begin{proof}
Since $\mathcal{S}^{\wedge_1,\wedge_2}_{\mathbb{H},g}f$ is $\epsilon-$concentrated on $\Omega,$ we have
\begin{align*}
\left\|\chi_{\Omega^c}\mathcal{S}^{\boldsymbol\alpha,\boldsymbol\beta}_{\mathbb{H},g}f\right\|^2_{L^2_\mathbb{H}(\mathbb{R}^4\times\mathbb{R}^4)}\leq \epsilon^2\|g\|^2_{L^2_\mathbb{H}(\mathbb{R}^4)}\|f\|^2_{L^2_\mathbb{H}(\mathbb{R}^4)}.
\end{align*}
This gives
\begin{align}\label{P7eqn47}
\left\|\chi_{\Omega}\mathcal{S}^{\boldsymbol\alpha,\boldsymbol\beta}_{\mathbb{H},g}f\right\|^2_{L^2_\mathbb{H}(\mathbb{R}^4\times\mathbb{R}^4)}\geq (1-\epsilon^2)\|g\|^2_{L^2_\mathbb{H}(\mathbb{R}^4)}\|f\|^2_{L^2_\mathbb{H}(\mathbb{R}^4)}.
\end{align}
By Holder's inequality, we get
\begin{align*}
\left\|\chi_{\Omega}\mathcal{S}^{\boldsymbol\alpha,\boldsymbol\beta}_{\mathbb{H},g}f\right\|^2_{L^2_\mathbb{H}(\mathbb{R}^4\times\mathbb{R}^4)}\leq &\left\{\int_{\mathbb{R}^4}\int_{\mathbb{R}^4}\left(\chi_{\Omega}(\bldx,\bldxi)\right)^{\frac{q}{q-2}}d\bldx d\bldxi\right\}^{\frac{q-2}{2}}\left\{\int_{\mathbb{R}^4}\int_{\mathbb{R}^4}\left(\left|\left(\mathcal{S}^{\boldsymbol\alpha,\boldsymbol\beta}_{\mathbb{H},g}f\right)(\bldx,\bldxi)\right|^2\right)^{\frac{q}{2}}d\bldx d\bldxi\right\}^{\frac{2}{q}}\\
&=|\Omega|^{\frac{q-2}{q}}\left\|\mathcal{S}^{\boldsymbol\alpha,\boldsymbol\beta}_{\mathbb{H},g}f\right\|^2_{L^q_\mathbb{H}(\mathbb{R}^4\times\mathbb{R}^4)}.
\end{align*}
Using, the Lieb's inequality \eqref{P7eqn41}, we get
\begin{align}\label{P7eqn48}
\left\|\chi_{\Omega}\mathcal{S}^{\boldsymbol\alpha,\boldsymbol\beta}_{\mathbb{H},g}f\right\|^2_{L^2_\mathbb{H}(\mathbb{R}^4\times\mathbb{R}^4)}\leq |\Omega|^{\frac{q-2}{q}}\left|\sin\left(\frac{\UVWX+\XWVU}{2}\right)\right|^{\frac{4}{q}-2}\left|\sin\left(\frac{\PQRS+\SRQP}{2}\right)\right|^{\frac{4}{q}-2} \left(\frac{2}{q}\right)^{\frac{8}{q}}\|f\|^2_{L^2_\mathbb{H}(\mathbb{R}^4)}\|g\|^2_{L^2_\mathbb{H}(\mathbb{R}^4)}.
\end{align}
From equation \eqref{P7eqn47} and equation \eqref{P7eqn48}, we get
\begin{align*}
 |\Omega|^{\frac{q-2}{q}}\left|\sin\left(\frac{\UVWX+\XWVU}{2}\right)\right|^{\frac{4}{q}-2}\left|\sin\left(\frac{\PQRS+\SRQP}{2}\right)\right|^{\frac{4}{q}-2} \left(\frac{2}{q}\right)^{\frac{8}{q}}\geq (1-\epsilon^2).
\end{align*}
This gives
\begin{align*}
&|\Omega|\geq \left(\left|\sin\left(\frac{\UVWX+\XWVU}{2}\right)\right|\left|\sin\left(\frac{\PQRS+\SRQP}{2}\right)\right|\right)^{2\left(1-\frac{2}{q}\right)\frac{q}{q-2}}(1-\epsilon^2)^{\frac{q}{q-2}}\left(\frac{q}{2}\right)^{\frac{8}{q-2}},~\mbox{since}~\frac{1}{p}+\frac{1}{q}=1\\
&\mbox{i.e.,}~|\Omega|\geq \left|\sin\left(\frac{\UVWX+\XWVU}{2}\right)\right|^2\left|\sin\left(\frac{\PQRS+\SRQP}{2}\right)\right|^2(1-\epsilon^2)^{\frac{q}{q-2}}\left(\frac{q}{2}\right)^{\frac{8}{q-2}}.
\end{align*}
This proves \eqref{P7Bound_of_EssentialSupportofOmega}.
\end{proof}
\begin{remark}
If $\epsilon=0,$ from \eqref{P7Bound_of_EssentialSupportofOmega}, we see that
\begin{align}
&\left|\rm{supp}\left(\mathcal{S}^{\boldsymbol\alpha,\boldsymbol\beta}_{\mathbb{H},g}f\right)\right|\geq \left|\sin\left(\frac{\UVWX+\XWVU}{2}\right)\right|^2\left|\sin\left(\frac{\PQRS+\SRQP}{2}\right)\right|^2\lim_{q\rightarrow 2+}\left(\frac{q}{2}\right)^{\frac{8}{q-2}}\notag\\
&\mbox{i.e.,}~\left|\rm{supp}\left(\mathcal{S}^{\boldsymbol\alpha,\boldsymbol\beta}_{\mathbb{H},g}f\right)\right|\geq \left|\sin\left(\frac{\UVWX+\XWVU}{2}\right)\right|^2\left|\sin\left(\frac{\PQRS+\SRQP}{2}\right)\right|^2e^4.
\end{align}
\end{remark}
i.e., measure of the supp $\mathcal{S}^{\boldsymbol\alpha,\boldsymbol\beta}_{\mathbb{H},g}f\geq \left|\sin\left(\frac{\UVWX+\XWVU}{2}\right)\right|^2\left|\sin\left(\frac{\PQRS+\SRQP}{2}\right)\right|^2e^4.$

\begin{remark} For $\boldsymbol\alpha=\boldsymbol\beta,$ the  theorem \ref{P7Theo4.8} gives the Lieb's UP for the STQFrFT of the function $f\in L^2_{\mathbb{H}}(\mathbb{R}^4).$ Moreover, if $\boldsymbol\alpha=\boldsymbol\beta=(\frac{\pi}{2},\frac{\pi}{2}),$ theorem \ref{P7Theo4.8} gives the Lieb's UP for the STQFT of the function $f\in L^2_{\mathbb{H}}(\mathbb{R}^4).$
\end{remark}

\subsection{Entropy uncertainty principle}
\begin{theorem}\label{P7Theo4.9}
Let $f\in L^2_\mathbb{H}(\mathbb{R}^4)$ and $g$ be a QWF with $\|f\|_{L^2_\mathbb{H}(\mathbb{R}^4)}\|g\|_{L^2_\mathbb{H}(\mathbb{R}^4)}=1,$ then
\begin{align}\label{P7eqn54}
\mathcal{E}_{S}(f,g,\boldsymbol\alpha,\boldsymbol\beta)\geq 2\left[2+\log\left(\left|\sin\left(\frac{\UVWX+\XWVU}{2}\right)\right|\left|\sin\left(\frac{\PQRS+\SRQP}{2}\right)\right|\right)\right],
\end{align}
where
$\displaystyle\mathcal{E}_{S}(f,g,\boldsymbol\alpha,\boldsymbol\beta)=-\int_{\mathbb{R}^4}\int_{\mathbb{R}^4}\left|\left(\mathcal{S}_{\mathbb{H},g}^{\boldsymbol\alpha,\boldsymbol\beta}f\right)(\bldx,\bldxi)\right|^2\log\left(\left|\left(\mathcal{S}_{\mathbb{H},g}^{\boldsymbol\alpha,\boldsymbol\beta}f\right)(\bldx,\bldxi)\right|^2\right)d\bldx d\bldxi.$
\end{theorem}
\begin{proof}
Define 
\begin{align}\label{P7eqn55}
I(f,g,\boldsymbol\alpha,\boldsymbol\beta,q)=\int_{\mathbb{R}^4}\int_{\mathbb{R}^4}\left|\left(\mathcal{S}_{\mathbb{H},g}^{\boldsymbol\alpha,\boldsymbol\beta}f\right)(\bldx,\bldxi)\right|^qd\bldx d\bldxi.
\end{align}
Then using  \eqref{P7eqn55} in \eqref{P7eqn37}, we get
\begin{align}\label{P7eqn56}
I(f,g,\boldsymbol\alpha,\boldsymbol\beta,2)=1.
\end{align}
Also, from \eqref{P7eqn41} and \eqref{P7eqn56}, it can be shown that
\begin{align}\label{P7eqn57}
I(f,g,\boldsymbol\alpha,\boldsymbol\beta,q)\leq \left(\left|\sin\left(\frac{\UVWX+\XWVU}{2}\right)\right|\left|\sin\left(\frac{\PQRS+\SRQP}{2}\right)\right|\right)^{2-q}\left(\frac{2}{q}\right)^4.
\end{align}
For $s>0,$ define
\begin{align*}
R(s)=\frac{I(f,g,\boldsymbol\alpha,\boldsymbol\beta,2)-I(f,g,\boldsymbol\alpha,\boldsymbol\beta,2+2s)}{s}.
\end{align*}
Then
\begin{align}\label{P7eqn58}
R(s)
&\geq\frac{1}{s}\left\{1-\left(\left|\sin\left(\frac{\UVWX+\XWVU}{2}\right)\right|\left|\sin\left(\frac{\PQRS+\SRQP}{2}\right)\right|\right)^{-2s}\left(\frac{1}{1+s}\right)^4\right\}.
\end{align}
Assume that $\mathcal{E}_{S}(f,g,\boldsymbol\alpha,\boldsymbol\beta)<\infty,$ otherwise \eqref{P7eqn54} is obvious.\\
Now from the inequality $1+s \log a\leq a^s,~s>0,$ we have
\begin{align}\label{P7eqn59}
0\leq\frac{1}{s}\left|\left(\mathcal{S}_{\mathbb{H},g}^{\boldsymbol\alpha,\boldsymbol\beta}f\right)(\bldx,\bldxi)\right|^2\left(1-\left|\left(\mathcal{S}_{\mathbb{H},g}^{\boldsymbol\alpha,\boldsymbol\beta}f\right)(\bldx,\bldxi)\right|^{2s}\right)\leq-\left|\left(\mathcal{S}_{\mathbb{H},g}^{\boldsymbol\alpha,\boldsymbol\beta}f\right)(\bldx,\bldxi)\right|^2\log\left(\left|\left(\mathcal{S}_{\mathbb{H},g}^{\boldsymbol\alpha,\boldsymbol\beta}f\right)(\bldx,\bldxi)\right|^2\right).
\end{align}
Since, $-\left|\left(\mathcal{S}_{\mathbb{H},g}^{\boldsymbol\alpha,\boldsymbol\beta}f\right)(\bldx,\bldxi)\right|^2\log\left(\left|\left(\mathcal{S}_{\mathbb{H},g}^{\boldsymbol\alpha,\boldsymbol\beta}f\right)(\bldx,\bldxi)\right|^2\right)$ is integrable, using Lebesgue dominated convergence theorem in \eqref{P7eqn59}, we get
\begin{align}\label{P7eqn60}
\lim_{s\rightarrow 0+}R(s)
&=\int_{\mathbb{R}^4}\int_{\mathbb{R}^4}\lim_{s\rightarrow 0+}\left\{\frac{1}{s}\left|\left(\mathcal{S}_{\mathbb{H},g}^{\boldsymbol\alpha,\boldsymbol\beta}f\right)(\bldx,\bldxi)\right|^2\left(1-\left|\left(\mathcal{S}_{\mathbb{H},g}^{\boldsymbol\alpha,\boldsymbol\beta}f\right)(\bldx,\bldxi)\right|^{2s}\right)\right\}d\bldx d\bldxi\notag\\
&=\mathcal{E}_{S}(f,g,\boldsymbol\alpha,\boldsymbol\beta).
\end{align} 
Again from \eqref{P7eqn58}, we get
\begin{align}\label{P7eqn61}
\lim_{s\rightarrow 0+}R(s)\geq2\left[2+\log\left(\left|\sin\left(\frac{\UVWX+\XWVU}{2}\right)\right|\left|\sin\left(\frac{\PQRS+\SRQP}{2}\right)\right|\right)\right].
\end{align}
Thus, by virtue of \eqref{P7eqn60} and \eqref{P7eqn61}, equation \eqref{P7eqn54} follows. This finishes the proof.
\end{proof}

\begin{remark} For $\boldsymbol\alpha=\boldsymbol\beta,$ the  theorem \ref{P7Theo4.9} gives the Entropy UP for the STQFrFT of the function $f\in L^2_{\mathbb{H}}(\mathbb{R}^4).$ Moreover, if $\boldsymbol\alpha=\boldsymbol\beta=(\frac{\pi}{2},\frac{\pi}{2}),$ theorem \ref{P7Theo4.9} gives the entropy UP for the STQFT of the function $f\in L^2_{\mathbb{H}}(\mathbb{R}^4).$
\end{remark}

\textbf{Example of STQCFrFT:}\\

Consider a function $f(\bldt)=e^{-\left(|\bldt_1|^2+|\bldt_2|^2\right)},~\bldt=(\bldt_1,\bldt_2)\in\mathbb{R}^2\times\mathbb{R}^2,$ with $\bldt_1=(\tee_1,\tee_2)$ and $\bldt_2=(\tee_3,\tee_4).$ Also consider a function 
$g(\bldt)=
\begin{cases}
1,&0\leq \tee_1<\frac{1}{2}, 0\leq \tee_2<\frac{1}{2}, 0\leq \tee_3<\frac{1}{2}, 0\leq \tee_4<\frac{1}{2}\\
-1,&\frac{1}{2}\leq \tee_1<1, \frac{1}{2}\leq \tee_2<1, \frac{1}{2}\leq \tee_3<1, \frac{1}{2}\leq \tee_4<1\\
0, &otherwise.
\end{cases}
$

Using the definition \ref{P7Defn4.1}, the STQCFrFT of $f$ with respect to the window function $g$ is given by 
\begin{align}\label{P7Example_Definition}
\left(\mathcal{S}^{\boldsymbol\alpha,\boldsymbol\beta}_{\mathbb{H},g}f\right)(\bldx,\bldxi)=\int_{\mathbb{R}^4}\mathcal{K}^i_{\UVWX,\XWVU}(\bldt_1,\bldxi_1)f(\bldt)\overline{g(\bldt-\bldx)}\mathcal{K}^j_{\PQRS,\SRQP}(\bldt_2,\bldxi_2)d\bldt,~(\bldx,\bldxi)\in\mathbb{R}^4\times\mathbb{R}^4,
\end{align}
where $\bldx=(\bldx_1,\bldx_2),~\bldxi=(\bldxi_1,\bldxi_2)\in\mathbb{R}^2\times\mathbb{R}^2$  with $\bldx_1=(\xx_1,\xx_2),~\bldx_2=(\xx_3,\xx_4)$ and $\bldxi_1=(\xxi_1,\xxi_2),~\bldxi_2=(\xxi_3,\xxi_4).$ Thus for the choose function $f$ and the window function $g$ we get from \eqref{P7Example_Definition}
\begin{align}\label{P7Example_Integral1}
\left(\mathcal{S}^{\boldsymbol\alpha,\boldsymbol\beta}_{\mathbb{H},g}f\right)&(\bldx,\bldxi)\notag\\
&=\left\{\int_{\xx_1}^{\xx_1+\frac{1}{2}}\int_{\xx_2}^{\xx_2+\frac{1}{2}}\tilde{d}(\PONM) e^{-i\left\{\tilde{a}(\PONM)(|\bldt_1|^2+|\bldxi_1|^2)-\bldt_1\cdot M_1\bldxi_1\right\}}e^{-|\bldt_1|^2}d\bldt_1\right\}\notag\\
&\hspace{5cm}\times\left\{\int_{\xx_3}^{\xx_3+\frac{1}{2}}\int_{\xx_4}^{\xx_4+\frac{1}{2}} \tilde{d}(\MNOP)e^{-j\left\{\tilde{a}(\MNOP)(|\bldt_2|^2+|\bldxi_2|^2)-\bldt_2\cdot M_2\bldxi_2\right\}}e^{-|\bldt_2|^2}d\bldt_2\right\}\notag\\
&-\left\{\int_{\xx_1+\frac{1}{2}}^{\xx_1+1}\int_{\xx_2+\frac{1}{2}}^{\xx_2+1} \tilde{d}(\PONM)e^{-i\left\{\tilde{a}(\PONM)(|\bldt_1|^2+|\bldxi_1|^2)-\bldt_1\cdot M_1\bldxi_1\right\}}e^{-|\bldt_1|^2}d\bldt_1\right\}\notag\\
&\hspace{5cm}\times\left\{\int_{\xx_3+\frac{1}{2}}^{\xx_3+1}\int_{\xx_4+\frac{1}{2}}^{\xx_4+1}\tilde{d}(\MNOP) e^{-j\left\{\tilde{a}(\MNOP)(|\bldt_2|^2+|\bldxi_2|^2)-\bldt_2\cdot M_2\bldxi_2\right\}}e^{-|\bldt_2|^2}d\bldt_2\right\}.
\end{align}
We first consider the integral 
\begin{align}\label{P7Example_Integral2}
&\int_{\xx_1}^{\xx_1+\frac{1}{2}}\int_{\xx_2}^{\xx_2+\frac{1}{2}}\tilde{d}(\PONM)e^{-i\left\{\tilde{a}(\PONM)(|\bldt_1|^2+|\bldxi_1|^2)-\bldt_1\cdot M_1\bldxi_1\right\}}e^{-|\bldt_1|^2}d\bldt_1\notag\\
&=\tilde{d}(\PONM)\left\{\int_{\xx_1}^{\xx_1+\frac{1}{2}}e^{-i\left\{\tilde{a}(\PONM)\left(\tee_1^2+\xxi_1^2\right)-\tee_1\left(\tilde{b}(\PONM,\DCBA)\xxi_1+\tilde{c}(\PONM,\DCBA)\xxi_2\right)\right\}}e^{-\tee_1^2}d\tee_1\right\}\notag\\
&\hspace{7cm}\times\left\{\int_{\xx_2}^{\xx_2+\frac{1}{2}}e^{-i\left\{\tilde{a}(\PONM)\left(\tee_2^2+\xxi_2^2\right)-\tee_2\left(-\tilde{c}(\PONM,\DCBA)\xxi_1+\tilde{b}(\PONM,\DCBA)\xxi_2\right)\right\}}e^{-\tee_2^2}d\tee_2\right\}.
\end{align}
Consider the integral
\begin{align}\label{P7Example_Integral3}
&\int_{\xx_1}^{\xx_1+\frac{1}{2}}e^{-i\left\{\tilde{a}(\PONM)\left(\tee_1^2+\xxi_1^2\right)-\tee_1\left(\tilde{b}(\PONM,\DCBA)\xxi_1+\tilde{c}(\PONM,\DCBA)\xxi_2\right)\right\}}e^{-\tee_1^2}d\tee_1\notag\\
&=e^{-i\tilde{a}(\PONM)\xxi_1^2}\int_{\xx_1}^{\xx_1+\frac{1}{2}}e^{-\left\{\left(1+i\tilde{a}(\PONM)\right)\tee_1^2-i\tee_1\left(\tilde{b}(\PONM,\DCBA)\xxi_1+\tilde{c}(\PONM,\DCBA)\xxi_2\right)\right\}}d\tee_1\notag\\
&=\frac{\sqrt{\pi}e^{-i\tilde{a}(\PONM)\xxi_1^2-\frac{\left(\tilde{b}(\PONM,\DCBA)\xxi_1+\tilde{c}(\PONM,\DCBA)\xxi_2\right)^2}{4\left(1+i\tilde{a}(\PONM)\right)}}}{2\sqrt{1+i\tilde{a}(\PONM)}}\Bigg[erf\left(A_1(\UVWX,\XWVU,i)\left(\xx_1+\frac{1}{2}\right)-B_1(\UVWX,\XWVU,i,\bldxi_1)\right)\notag\\
&\hspace{10cm}-erf\left(A_1(\UVWX,\XWVU,i)\xx_1-B_1(\UVWX,\XWVU,i,\bldxi_1)\right)\Bigg],
\end{align}
where $erf(\xx)=\frac{2}{\sqrt{\pi}}\int_0^{\xx}e^{-t^2}dt,$ $A(\UVWX,\XWVU,i)=\sqrt{1+i\tilde{a}(\PONM)}$ and $B_1(\UVWX,\XWVU,i)=\frac{\left(\tilde{b}(\PONM,\DCBA)\xxi_1+\tilde{c}(\PONM,\DCBA)\xxi_2\right)^2}{4\left(1+i\tilde{a}(\PONM)\right)}.$ Similarly, we have
\begin{align}\label{P7Example_Integral4}
&\int_{\xx_2}^{\xx_2+\frac{1}{2}}e^{-i\left\{\tilde{a}(\PONM)\left(\tee_2^2+\xxi_2^2\right)-\tee_2\left(-\tilde{c}(\PONM,\DCBA)\xxi_1+\tilde{b}(\PONM,\DCBA)\xxi_2\right)\right\}}e^{-\tee_2^2}d\tee_2\notag\\
&=\frac{\sqrt{\pi}e^{-i\tilde{a}(\PONM)\xxi_2^2-\frac{\left(-\tilde{c}(\PONM,\DCBA)\xxi_1+\tilde{b}(\PONM,\DCBA)\xxi_2\right)^2}{4\left(1+i\tilde{a}(\PONM)\right)}}}{2\sqrt{1+i\tilde{a}(\PONM)}}\Bigg[erf\left(A(\UVWX,\XWVU,i)\left(\xx_2+\frac{1}{2}\right)-B_2(\UVWX,\XWVU,i,\bldxi_1)\right)\notag\\
&\hspace{10cm}-erf\left(A(\UVWX,\XWVU,i)\xx_2-B_2(\UVWX,\XWVU,i,\bldxi_1)\right)\Bigg],
\end{align}
where $B_2(\UVWX,\XWVU,i)=\frac{\left(-\tilde{c}(\PONM,\DCBA)\xxi_1+\tilde{b}(\PONM,\DCBA)\xxi_2\right)^2}{4\left(1+i\tilde{a}(\PONM)\right)}.$ Using equations \eqref{P7Example_Integral3} and \eqref{P7Example_Integral4} in \eqref{P7Example_Integral2}, we get
\begin{align}\label{P7Example_Integral5}
\int_{\xx_1}^{\xx_1+\frac{1}{2}}&\int_{\xx_2}^{\xx_2+\frac{1}{2}}\tilde{d}(\PONM)e^{-i\left\{\tilde{a}(\PONM)(|\bldt_1|^2+|\bldxi_1|^2)-\bldt_1\cdot M_1\bldxi_1\right\}}e^{-|\bldt_1|^2}d\bldt_1\notag\\
&=J(\UVWX,\XWVU,i,\bldxi_1)\left[erf\left(A(\UVWX,\XWVU,i)\left(\xx_1+\frac{1}{2}\right)-B_1(\UVWX,\XWVU,i,\bldxi_1)\right)-erf\left(A(\UVWX,\XWVU,i)\xx_1-B_1(\UVWX,\XWVU,i,\bldxi_1)\right)\right]\notag\\
&\times \left[erf\left(A(\UVWX,\XWVU,i)\left(\xx_2+\frac{1}{2}\right)-B_2(\UVWX,\XWVU,i,\bldxi_1)\right)-erf\left(A(\UVWX,\XWVU,i)\xx_2-B_2(\UVWX,\XWVU,i,\bldxi_1)\right)\right],
\end{align}
where $J(\UVWX,\XWVU,i,\bldxi_1)=\tilde{d}(\PONM)\frac{\pi}{4(1+i\tilde{a}(\PONM))}e^{-\left\{i\tilde{a}(\PONM)|\bldxi_1|^2+\frac{\left((\tilde{b}(\PONM,\DCBA))^2+(\tilde{c}(\PONM,\DCBA))^2\right)|\bldxi_1|^2}{4(1+i\tilde{a}(\PONM))}\right\}}.$ Similarly, we have
\begin{align}\label{P7Example_Integral6}
\int_{\xx_3}^{\xx_3+\frac{1}{2}}&\int_{\xx_4}^{\xx_4+\frac{1}{2}}\tilde{d}(\MNOP)e^{-j\left\{\tilde{a}(\MNOP)(|\bldt_2|^2+|\bldxi_2|^2)-\bldt_2\cdot M_2\bldxi_2\right\}}e^{-|\bldt_2|^2}d\bldt_1\notag\\
&=J(\PQRS,\SRQP,j,\bldxi_2)\left[erf\left(A(\PQRS,\SRQP,j)\left(\xx_3+\frac{1}{2}\right)-B_1(\PQRS,\SRQP,j,\bldxi_2)\right)-erf\left(A(\PQRS,\SRQP,j)\xx_3-B_1(\PQRS,\SRQP,j,\bldxi_2)\right)\right]\notag\\
&\times \left[erf\left(A(\PQRS,\SRQP,j)\left(\xx_4+\frac{1}{2}\right)-B_2(\PQRS,\SRQP,j,\bldxi_2)\right)-erf\left(A(\PQRS,\SRQP,j)\xx_4-B_2(\PQRS,\SRQP,j,\bldxi_2)\right)\right].
\end{align}
Thus from equations \eqref{P7Example_Integral5} and \eqref{P7Example_Integral6}, we have
\begin{align}\label{P7Example_Integral7}
&\left\{\int_{\xx_1}^{\xx_1+\frac{1}{2}}\int_{\xx_2}^{\xx_2+\frac{1}{2}}\tilde{d}(\PONM) e^{-i\left\{\tilde{a}(\PONM)(|\bldt_1|^2+|\bldxi_1|^2)-\bldt_1\cdot M_1\bldxi_1\right\}}e^{-|\bldt_1|^2}d\bldt_1\right\} \notag\\
&\hspace{6cm}\times\left\{\int_{\xx_3}^{\xx_3+\frac{1}{2}}\int_{\xx_4}^{\xx_4+\frac{1}{2}} \tilde{d}(\MNOP)e^{-j\left\{\tilde{a}(\MNOP)(|\bldt_2|^2+|\bldxi_2|^2)-\bldt_2\cdot M_2\bldxi_2\right\}}e^{-|\bldt_2|^2}d\bldt_2\right\}\notag\\
&=J(\UVWX,\XWVU,i,\bldxi_1)\left[erf\left(A(\UVWX,\XWVU,i)\left(\xx_1+\frac{1}{2}\right)-B_1(\UVWX,\XWVU,i,\bldxi_1)\right)-erf\left(A(\UVWX,\XWVU,i)\xx_1-B_1(\UVWX,\XWVU,i,\bldxi_1)\right)\right]\notag\\
&\times \left[erf\left(A(\UVWX,\XWVU,i)\left(\xx_2+\frac{1}{2}\right)-B_2(\UVWX,\XWVU,i,\bldxi_1)\right)-erf\left(A(\UVWX,\XWVU,i)\xx_2-B_2(\UVWX,\XWVU,i,\bldxi_1)\right)\right] J(\PQRS,\SRQP,j,\bldxi_2)\notag\\
&\times\left[erf\left(A(\PQRS,\SRQP,j)\left(\xx_3+\frac{1}{2}\right)-B_1(\PQRS,\SRQP,j,\bldxi_2)\right)-erf\left(A(\PQRS,\SRQP,j)\xx_3-B_1(\PQRS,\SRQP,j,\bldxi_2)\right)\right]\notag\\
&\times \left[erf\left(A(\PQRS,\SRQP,j)\left(\xx_4+\frac{1}{2}\right)-B_2(\PQRS,\SRQP,j,\bldxi_2)\right)-erf\left(A(\PQRS,\SRQP,j)\xx_4-B_2(\PQRS,\SRQP,j,\bldxi_2)\right)\right].
\end{align}
Similarly, it can be shown that 
\begin{align}\label{P7Example_Integral8}
&\left\{\int_{\xx_1+\frac{1}{2}}^{\xx_1+1}\int_{\xx_2+\frac{1}{2}}^{\xx_2+1} \tilde{d}(\PONM)e^{-i\left\{\tilde{a}(\PONM)(|\bldt_1|^2+|\bldxi_1|^2)-\bldt_1\cdot M_1\bldxi_1\right\}}e^{-|\bldt_1|^2}d\bldt_1\right\} \notag\\
&\hspace{6cm}\times\left\{\int_{\xx_3+\frac{1}{2}}^{\xx_3+1}\int_{\xx_4+\frac{1}{2}}^{\xx_4+1}\tilde{d}(\MNOP) e^{-j\left\{\tilde{a}(\MNOP)(|\bldt_2|^2+|\bldxi_2|^2)-\bldt_2\cdot M_2\bldxi_2\right\}}e^{-|\bldt_2|^2}d\bldt_2\right\}\notag\\
&=J(\UVWX,\XWVU,i,\bldxi_1)\left[erf\left(A(\UVWX,\XWVU,i)\left(\xx_1+1\right)-B_1(\UVWX,\XWVU,i,\bldxi_1)\right)-erf\left(A(\UVWX,\XWVU,i)\left(\xx_1+\frac{1}{2}\right)-B_1(\UVWX,\XWVU,i,\bldxi_1)\right)\right]\notag\\
&\times \left[erf\left(A(\UVWX,\XWVU,i)\left(\xx_2+1\right)-B_2(\UVWX,\XWVU,i,\bldxi_1)\right)-erf\left(A(\UVWX,\XWVU,i)\left(\xx_2+\frac{1}{2}\right)-B_2(\UVWX,\XWVU,i,\bldxi_1)\right)\right] J(\PQRS,\SRQP,j,\bldxi_2)\notag\\
&\times\left[erf\left(A(\PQRS,\SRQP,j)\left(\xx_3+1\right)-B_1(\PQRS,\SRQP,j,\bldxi_2)\right)-erf\left(A(\PQRS,\SRQP,j)\left(\xx_3+\frac{1}{2}\right)-B_1(\PQRS,\SRQP,j,\bldxi_2)\right)\right]\notag\\
&\times \left[erf\left(A(\PQRS,\SRQP,j)\left(\xx_4+1\right)-B_2(\PQRS,\SRQP,j,\bldxi_2)\right)-erf\left(A(\PQRS,\SRQP,j)\left(\xx_4+\frac{1}{2}\right)-B_2(\PQRS,\SRQP,j,\bldxi_2)\right)\right].
\end{align}
Thus from equations \eqref{P7Example_Definition}, \eqref{P7Example_Integral7} and \eqref{P7Example_Integral8} we obtain
\begin{align*}
&\left(\mathcal{S}^{\boldsymbol\alpha,\boldsymbol\beta}_{\mathbb{H},g}f\right)(\bldx,\bldxi)\\
&=J(\UVWX,\XWVU,i,\bldxi_1)\left[erf\left(A(\UVWX,\XWVU,i)\left(\xx_1+\frac{1}{2}\right)-B_1(\UVWX,\XWVU,i,\bldxi_1)\right)-erf\left(A(\UVWX,\XWVU,i)\xx_1-B_1(\UVWX,\XWVU,i,\bldxi_1)\right)\right]\notag\\
&\times \left[erf\left(A(\UVWX,\XWVU,i)\left(\xx_2+\frac{1}{2}\right)-B_2(\UVWX,\XWVU,i,\bldxi_1)\right)-erf\left(A(\UVWX,\XWVU,i)\xx_2-B_2(\UVWX,\XWVU,i,\bldxi_1)\right)\right] J(\PQRS,\SRQP,j,\bldxi_2)\notag\\
&\times\left[erf\left(A(\PQRS,\SRQP,j)\left(\xx_3+\frac{1}{2}\right)-B_1(\PQRS,\SRQP,j,\bldxi_2)\right)-erf\left(A(\PQRS,\SRQP,j)\xx_3-B_1(\PQRS,\SRQP,j,\bldxi_2)\right)\right]\notag\\
&\times \left[erf\left(A(\PQRS,\SRQP,j)\left(\xx_4+\frac{1}{2}\right)-B_2(\PQRS,\SRQP,j,\bldxi_2)\right)-erf\left(A(\PQRS,\SRQP,j)\xx_4-B_2(\PQRS,\SRQP,j,\bldxi_2)\right)\right]\notag\\
&-J(\UVWX,\XWVU,i,\bldxi_1)\left[erf\left(A(\UVWX,\XWVU,i)\left(\xx_1+1\right)-B_1(\UVWX,\XWVU,i,\bldxi_1)\right)-erf\left(A(\UVWX,\XWVU,i)\left(\xx_1+\frac{1}{2}\right)-B_1(\UVWX,\XWVU,i,\bldxi_1)\right)\right]\notag\\
&\times \left[erf\left(A(\UVWX,\XWVU,i)\left(\xx_2+1\right)-B_2(\UVWX,\XWVU,i,\bldxi_1)\right)-erf\left(A(\UVWX,\XWVU,i)\left(\xx_2+\frac{1}{2}\right)-B_2(\UVWX,\XWVU,i,\bldxi_1)\right)\right] J(\PQRS,\SRQP,j,\bldxi_2)\notag\\
&\times\left[erf\left(A(\PQRS,\SRQP,j)\left(\xx_3+1\right)-B_1(\PQRS,\SRQP,j,\bldxi_2)\right)-erf\left(A(\PQRS,\SRQP,j)\left(\xx_3+\frac{1}{2}\right)-B_1(\PQRS,\SRQP,j,\bldxi_2)\right)\right]\notag\\
&\times \left[erf\left(A(\PQRS,\SRQP,j)\left(\xx_4+1\right)-B_2(\PQRS,\SRQP,j,\bldxi_2)\right)-erf\left(A(\PQRS,\SRQP,j)\left(\xx_4+\frac{1}{2}\right)-B_2(\PQRS,\SRQP,j,\bldxi_2)\right)\right].
\end{align*}
 
\section{Conclusions}
In this paper, we have defined the two sided QCFrFT and obtained the Parseval's formula and the sharp Hausdorff-Young inequality based on which we obtained its R\`enyi entropy UP. Including the basic properties like boundedness, translation, etc of the newly proposed STQCFrFT, we also obtain the inner product relation followed by the reconstruction formula. Not only that with the aid of sharp Hausdorff-Young inequality for the QCFrFT we also obtained the Lieb's and entropy UPs for the proposed STQCFrFT.

\section{Acknowledgement}
This work is partially supported by UGC File No. 16-9(June 2017)/2018(NET/CSIR), New Delhi, India.
\bibliography{P7MasterB7_CouQFrFT}
\bibliographystyle{plain}
\end{document}